\documentclass{amsproc}
\usepackage{euscript, graphicx, epstopdf}
\usepackage{cases}
\usepackage{mathrsfs}
\usepackage{bbm}
\usepackage{amssymb}
\usepackage{txfonts}
\usepackage{amscd}
\usepackage{amsfonts,latexsym,amsmath,amsxtra,mathdots,amssymb,latexsym,mathabx}
\usepackage[all,cmtip]{xy}
\usepackage{color}
\usepackage{colordvi}
\usepackage{multicol}
\usepackage{hyperref}
\usepackage{tikz}
\usepackage{float}
\usepackage{setspace, floatflt}

\allowdisplaybreaks

\newcommand{\Tr}{{{\rm Tr}}}

\def\mod{\mathrm{mod}\ }

\newcommand{\BC}{{\mathbb {C}}}

\newcommand{\BQ}{{\mathbb {Q}}} \newcommand{\BR}{{\mathbb {R}}}

 \newcommand{\BZ}{{\mathbb {Z}}}

\newcommand{\depth}{{\mathrm{depth}}}

 \renewcommand{\Im}{{\mathrm{Im}\shskip }}

\newcommand{\PGL}{{\mathrm{PGL}}} 
\renewcommand{\Re}{{\mathrm{Re}\shskip }}

\newcommand{\PSL}{{\mathrm{PSL}}}

\def\-{^{-1}}

\def\-{^{-1}}

\def\lp {\left (}
\def\rp {\right )}

\def\SSS{\text{\mbox{\larger[1]$\text{\usefont{U}{BOONDOX-cal}{m}{n}S}$}}\hskip 1pt}
\def\SS{\raisebox{- 2 \depth}{$\SSS$}}
\def\bfJ {\boldsymbol J}
\def\bfG{\boldsymbol{G}}
\def\bfF{{\boldsymbol {F}}}
\def\bfA{{\boldsymbol{A}}}
\def\bfB{{\boldsymbol{B}}}
\def\bfC{{\boldsymbol{C}}}
\def\bfD{{\boldsymbol {D}}}
\def\bfE{{\boldsymbol {E}}}

\def\nwedge {\hskip - 2 pt \wedge \hskip - 2 pt }
\def\shskip{\hskip 0.5 pt}

\makeatletter
\g@addto@macro\normalsize{\setlength\abovedisplayskip{3pt}}
\makeatother

\makeatletter
\g@addto@macro\normalsize{\setlength\belowdisplayskip{3pt}}
\makeatother

\newcommand{\delete}[1]{}

 \newcommand{\SL}{{\mathrm{SL}}}

\newcommand{\sstyle}{\scriptstyle}

\newcommand{\ra}{\rightarrow}

\theoremstyle{plain}

\newtheorem{thm}{Theorem}[section] \newtheorem{cor}[thm]{Corollary}
\newtheorem{lem}[thm]{Lemma}  \newtheorem{prop}[thm]{Proposition}

\newtheorem {rem}[thm]{Remark}

\newtheorem*{acknowledgement}{Acknowledgements}

\numberwithin{equation}{section}

\begin{document}

	\title[On the Fourier Transform of Bessel Functions over Complex Numbers---II]{On the Fourier Transform of Bessel Functions over Complex Numbers---II: the General Case}
	
	\author{Zhi Qi}
	\address{School of Mathematical Sciences\\ Zhejiang University\\Hangzhou, 310027\\China}
	\email{zhi.qi@zju.edu.cn}
	

	\begin{abstract}
		In this paper, we prove an exponential integral formula for the  Fourier transform of Bessel functions over complex numbers, along with a radial  exponential integral formula. The former will enable us to develop the complex  spectral theory of the relative trace formula for the Shimura-Waldspurger correspondence and extend the Waldspurger formula from totally real fields to arbitrary number fields. 
	\end{abstract}
	
	\maketitle

\section{Introduction}

\subsection{Representation Theoretic Motivations}
It is known by the work of Baruch and Mao (\cite{BaruchMao-Real,BaruchMao-Global}) that 
the exponential integral formulae due to Weber and Hardy on the Fourier transform of classical Bessel functions over real numbers realize the Shimura-Waldspurger correspondence between representations  of $\PGL_2 (\BR)$ and genuine representations of $\widetilde{\SL}_2 (\BR)$ and constitute the real component of the Waldspurger formula for automorphic forms of $\PGL_2  $ and $\widetilde \SL_2  $ over $\BQ$ or a totally real field. For instance, the formula of Weber is as follows
\begin{equation}\label{0eq: Weber's formula}
\int_0^\infty \frac 1 {\sqrt x}  J_{\nu} \lp 4 \pi \sqrt x\rp e \lp {\pm x y} \rp   {d x}  = \frac 1 {\sqrt {2 y} } e \lp {\mp  \lp \frac 1 {2 y} - \frac 1 8  \nu - \frac 1 8 \rp} \rp J_{\frac 1 2 \nu} \lp \frac {\pi} y \rp,
\end{equation}
for $y > 0$,  where $e(x) = \exp \lp {2\pi i x} \rp$ and $J_\nu (x)$ is the  Bessel function of the first kind of order $\nu$. This formula is valid when  $\Re \nu > - 1$. Taking $\nu = 2 k - 1$  in \eqref{0eq: Weber's formula}, with $k$ a positive integer, 
the Bessel function of order $2 k - 1$, respectively $k - \frac 1 2$, is attached to a discrete series representation of $\PGL_2 (\BR)$, respectively $\widetilde {\SL}_2 (\BR)$. Thus,  in this case, \eqref{0eq: Weber's formula} should be interpreted as the local ingredient at the real place of the correspondence due to  Shimura, Shintani and Waldspurger between cusp forms of weight $2 k$ and cusp forms of weight $ k + \frac 1 2$. 

The purpose of this paper is to prove the complex analogue of Weber and Hardy's formulae for Bessel functions over complex numbers. 
As applications of this paper in the future, one may develop the complex  spectral theory of the relative trace formula for the Shimura-Waldspurger correspondence as the real theory in \cite{BaruchMao-Real}, and furthermore extend the Waldspurger formula from totally real fields as in \cite{BaruchMao-Global} to arbitrary number fields.

\subsection{Statement of Results}


We now introduce the definition of Bessel functions over complex numbers (see \cite[\S 15.3]{Qi-Bessel}, \cite[(6.21), (7.21)]{B-Mo}). Let $\mu  $ be a complex number and $m  $ be an integer. 
We define 
\begin{equation}\label{0def: J mu m (z)}
J_{\mu,\shskip  m} (z) = J_{- 2\mu - \frac 12 m } \lp  z \rp J_{- 2\mu + \frac 12 m  } \lp  {\overline z} \rp.
\end{equation}
The function $J_{\mu,\shskip  m} (z)$ is well defined in the sense that the   expression on the right of \eqref{0def: J mu m (z)} is independent on the choice of the argument of $z$ modulo $2 \pi$. Next, we define
\begin{equation}\label{0eq: defn of Bessel}
\bfJ_{ \mu,\shskip  m} \lp z \rp = 
\left\{ 
\begin{split}
& \frac {2 \pi^2} {\sin (2\pi \mu)} \lp J_{\mu,\shskip  m} (4 \pi \sqrt z) -  J_{-\mu,\shskip  -m} (4 \pi \sqrt z) \rp \hskip 5 pt \text {if } m \text{ is even},\\
& \frac {2 \pi^2 i} {\cos (2\pi \mu)} \lp J_{\mu,\shskip  m} (4 \pi \sqrt z) + J_{-\mu,\shskip   -m} (4 \pi \sqrt z) \rp \hskip 5 pt \text {if }  m \text{ is odd},
\end{split}
\right.
\end{equation} 
where $\sqrt z$ is the principal branch of the square root, and it is understood that in the nongeneric case when  $4 \mu \in 2\BZ + m$ the right hand side should be replaced by its limit.
We stress that  $\bfJ_{ \mu,\shskip   m} \lp z \rp$ is well defined only when $m$ is even; nevertheless $ \bfJ_{ \mu,\shskip   m}   (z^2)   $ is always a well defined function on the complex plane. Moreover, we note that $ \bfJ_{ - \mu,\shskip  - m} \lp z \rp = \bfJ_{ \mu,\shskip   m} \lp z \rp $, so we may assume  with no loss of generality that $m$ is nonnegative.

\begin{rem}
	According to \cite[\S 17, 18]{Qi-Bessel}, on choosing the Weyl element $ \begin{pmatrix}
	 & - 1 \\
	 1 &
	\end{pmatrix}$, when $m$ is even, respectively odd, $|z| \bfJ_{\mu,\shskip   m} (z)$, respectively $ \sqrt{ |z| \overline z} \bfJ_{\mu,\shskip   m} (z)$, is the Bessel function associated with the principal series representation $\pi_{\mu, \shskip  m}$ of $\SL_2 (\BC)$ {\rm(}not necessarily unitary{\rm)} induced from the character $\chiup_{\shskip \mu,\shskip  m}\begin{pmatrix}
	a & \\
	 & a\- 
	\end{pmatrix} = |a|^{4 \mu} (a/|a|)^{  m}$. In the even case, the principal series is indeed a representation of $\PGL_2 (\BC)${\rm(}$= \PSL_2 (\BC)${\rm)}. 
	
	For the kernel formula that defines Bessel functions for $\SL_2 (\BC)$ in representation theory, its proof and applications,  see {\rm \cite{B-Mo-Kernel2,Mo-Kernel2,Baruch-Kernel,Qi-Bessel,Qi-Kuz}}.
\end{rem}


Our main theorem  is as follows.

\begin{thm}\label{thm: main 1} 
	 Suppose that $|\Re \mu| < \frac 1 2$ and $m$ is even. 
	We have the identity
	\begin{equation}\label{eq: main 1}
	\begin{split}
	\int_{0}^{2 \pi} \int_0^\infty \bfJ_{ \mu,\shskip  m } \hskip - 2 pt \lp x e^{i\phi} \rp  e (- 2 x y \cos (\phi + \theta) )    d x  d \phi  
	= \frac {1} {4 y}   e\lp \frac {\cos \theta } y \rp  \bfJ_{  \frac 12 \mu, \frac 1 2 m } \lp  \frac 1 {
		16 y^2  e^{2 i \theta}  } \rp ,
	\end{split}
	\end{equation}
	for $y \in (0, \infty)$ and $\theta \in [0, 2 \pi)$. 
\end{thm}

\begin{rem}\label{rem: Waldspurger}
	The identity \eqref{eq: main 1} reflects  the  Shimura-Waldspurger correspondence between the principal series $\pi_{\mu, \shskip  m} $ of $\PGL_2 (\BC)$ and the principal series $\pi_{\frac 1 2 \mu, \shskip  \frac 1 2 m}$ of  $\SL_2 (\BC)$.
	It should be noted that, unlike $\SL_2 (\BR)$, there is no nontrivial double cover of $\SL_2 (\BC)$ and there do not exist discrete series for  $\SL_2 (\BC)$. 
\end{rem}

In  a previous paper \cite{Qi-Sph},  using two formulae for classical Bessel functions, the author has proved \eqref{eq: main 1} in the spherical case when $m = 0$. For the nonspherical case, it seems however that a straightforward proof as in \cite{Qi-Sph} is almost impossible.  In this paper, our proof of \eqref{eq: main 1} is in an indirect manner and splits into two steps. 

In the first step, we shall prove a radial exponential integral formula (see \eqref{0eq: int of J = IK} in Theorem \ref{thm: integral of J} below), which is considered  weaker than  the formula \eqref{eq: main 1}, on the integral of the Bessel function $\bfJ_{ \mu,\shskip  m } \lp x e^{i\phi} \rp$ against the radial exponential function $\exp (- 2 \pi c x)$, instead of the Fourier kernel $e (- 2 x y \cos (\phi + \theta) )$. Interestingly, it turns out that the bulk of its proof is combinatorial. 

In the second step, we shall prove Theorem \ref{thm: main 1} by exploiting a soft   method that combines  asymptotic analysis of oscillatory integrals and a uniqueness result for ordinary differential equations. The weak exponential integral formula \eqref{0eq: int of J = IK} is used to determine the constant term in the asymptotic, whereas the method of stationary phase for double integrals is applied for the oscillatory term. 

\begin{thm}\label{thm: integral of J}
	Suppose that $|\Re \mu| < \frac 1 2$ and $m$ is even. We have
	\begin{equation}\label{0eq: int of J = IK}
	\int_0^{2 \pi} \int_0^\infty \bfJ_{\mu,\shskip  m} \lp x e^{i \phi} \rp \exp (- 2 \pi c x)  d x d \phi = \frac { 4 \pi i^m} {    c   }  K_{2 \mu} \lp \frac {4 \pi} {c} \rp   I_{\frac 1 2 m} \lp \frac {4 \pi  } {c} \rp,
	\end{equation}
	for $|\arg c | < \frac 1 2 \pi$, where $I_{\nu} (z)$ and $K_{\nu} (z)$  are the two kinds of modified Bessel functions of order $\nu$.
\end{thm}

Although it is  not visible in the statement, there is a remarkable distinction between the spherical and nonspherical cases in the proof of Theorem \ref{thm: integral of J}. It comes from Kummer's confluent hypergeometric function $M  \lp \frac 1 2 m  + 1; m+1; z \rp$ arising in the proof (see \S \ref{sec: Kummer} and \S \ref{sec: reduction}), and makes the proof of the nonspherical case considerably harder. As alluded to above, when $ m = 2k \geqslant 2$, the identity \eqref{0eq: int of J = IK} in Theorem \ref{thm: integral of J} may be reduced to a complicated combinatorial recurrence identity as follows, 
\begin{align*} 
\sum_{n = 0}^{k}  (-)^n   C_{2k}^{k - n}   \sum_{r = 0}^{\left\lfloor   n/2 \right\rfloor} (-)^r (2/a)^{n-2r}  \big( C_{n-r}^r + C_{n-r-1}^{r-1} \big) \shskip  \partial_{\varv}^{ n-2r } \big( I_n (a \varv) \varv^k (1-\varv)^{1-k} \big) = 0,
\end{align*}
where $C^r_n$ denotes the binomial coefficient. Moreover, we remark that, when searching for a straightforward proof of  Theorem \ref{thm: main 1}, such a distinction persists and makes our attempts rather hopeless. 

	

\vskip 7 pt

Finally, we would also like to interpret Theorem \ref{thm: main 1} in the theory of distributions. 
Let   $  \SS (\BC)$ denote the space of Schwartz functions on $\BC$, that is, smooth functions  on $\BC$ that rapidly decay at infinity along with all of their derivatives. 
If rapid decay also occurs at zero, then we say the functions are   Schwartz functions on $ \BC \smallsetminus \{0\} $, and the space of such functions is denoted by  $\SS (\BC \smallsetminus \{0\})  $.

The Fourier transform $\widehat f  $ of a Schwartz function $f \in \SS(\BC)$ is defined by
\begin{equation*}
\widehat f (u) = \sideset{ }{_\BC }{\iint} \shskip  f (z) e(- \Tr  (uz) ) \shskip  i d z   \nwedge   d \overline z,
\end{equation*}
with $\Tr (z) = z + \overline z$. We have $\widehat{\widehat{f}} (z) = f (- z) $.

\begin{cor}\label{cor: main} Let $\mu$ be a complex number and $m$ be an even integer.
	We have
	\begin{equation}\label{eq: main 2}
	\begin{split}
	\sideset{ }{_{\BC \smallsetminus \{0\}} }{\iint}  \hskip - 2 pt   \bfJ_{\mu,\shskip  m} \lp z \rp    \widehat f (z)  \frac {i d z   \nwedge   d \overline z} {  {|z|}}   = \frac  1 2 \sideset{ }{_{\BC \smallsetminus \{0\}} }{\iint}     e \bigg( \Tr  \bigg( \frac 1 {2 u} \bigg) \bigg) \hskip - 1 pt \bfJ_{\frac 1 2 \mu,\shskip  \frac 1 2 m} \bigg(   \frac 1  { 16 u ^{ 2} } \bigg) \hskip - 1 pt f (u) \frac {i d u  \nwedge  d \overline u} {  {|u|}},
	\end{split}
	\end{equation}
	for all $ f \in \SS (\BC )$, under the assumption  $|\Re \mu | < \frac 1 2$. Furthermore, \eqref{eq: main 2} remains valid for all values of $\mu$ if one assumes $\widehat f \in \SS (\BC \smallsetminus \{0\})$.
\end{cor}

\subsection{An Application in Representation Theory: the Bessel Identity over the Complex Field}

Let 
\begin{align*}
N = \left\{ \begin{pmatrix}
1 & z \\ & 1 
\end{pmatrix} : z \in \BC \right\}, \hskip 10 pt A =  \left\{ \begin{pmatrix}
a &   \\ & a\-
\end{pmatrix} : a  \in \BC \smallsetminus \{0\} \right\}.
\end{align*} 
Let $\psi (z) = e (\Tr \shskip  z )$, viewed  as a character on $N$. 
Let $\pi$ be an  infinite-dimensional unitary irreducible representation of $\SL_2 (\BC)$. We attach to  $\pi$   a certain function $j_{\pi, \hskip 0.5 pt \psi}$    on $\SL_2 (\BC)$   which    is both left and right $(\psi, N)$-equivariant. When $\pi$ is trivial on the central (that is, a representation of $\PSL_2 (\BC) = \PGL_2 (\BC)$), we attach  a function $ i_{\hskip 0.5 pt \pi, \hskip 0.5 pt \psi} $  which  is left $A$-invariant  and right $(\psi, N)$-equivariant. $j_{\pi, \hskip 0.5 pt \psi}$ and $ i_{\hskip 0.5 pt \pi, \hskip 0.5 pt \psi} $  are called the Bessel function and the relative Bessel function for $\pi$ respectively. We stress that $\pi$ is determined by either of  these two functions.

As a consequence of Corollary \ref{cor: main}, we have the following Bessel identity. 
\begin{thm} Let   $\pi $ be the principal series $\pi_{\mu, \shskip  m} $ of $\PGL_2 (\BC)${\rm($=\PSL_2 (\BC)$)} and $\sigma$ be the principal series $\pi_{\frac 1 2 \mu, \shskip  \frac 1 2 m}$ of  $\SL_2 (\BC)$ {\rm(}see Remark {\rm\ref{rem: Waldspurger}}{\rm)}. For $z \in \BC \smallsetminus \{0\}$, we have
	\begin{align}
	i_{\hskip 0.5 pt \pi, \hskip 0.5 pt \psi} \begin{pmatrix}
	z / 4 & 1 \\ 1 & 
	\end{pmatrix} 
	= \frac {2 \epsilon (\pi, 1/2) \shskip  \psi     \lp   2 / z \rp |z|} {L(\pi, 1/2)}  j_{\sigma, \hskip 0.5 pt \psi} \begin{pmatrix}
	& - z\- \\ z \\ 
	\end{pmatrix}, 
	\end{align}
	in which $L (\pi, 1/2)$ and  $\epsilon (\pi, 1/2)$ are the central values of the $L$-factor and the $\epsilon$-factor associated with $ \pi $. 
\end{thm}

This is the complex analogue of \cite[Theorem 1.1]{BaruchMao-Real} and, along with the real and non-Archimedean Bessel identities in \cite{BaruchMao-Real} and \cite{BaruchMao-NA}, may be used to establish the Waldspurger formula over an arbitrary number field. These however do not seem to fit in the analytic theme of this paper and will be presented in two forthcoming papers   in collaboration with Jingsong Chai.\footnote{These were done very recently while this paper was under peer review. See \cite{Chai-Qi-Bessel} and \cite{Chai-Qi-Wald}.}


\section{Preliminaries}

\subsection{Classical Bessel Functions}

\subsubsection{Basic Properties of  $J_{\nu} (z)$, $H^{(1 )}_{\nu}  (z) $ and $H^{(2)}_{\nu}  (z) $} Let $\nu $ be a complex number.
Let $J_{\nu} (z)$, $H^{(1,\shskip  2)}_{\nu}  (z) $ denote the   Bessel function of the first kind and the Hankel functions of order $\nu$. They all satisfy the Bessel equation
\begin{equation}\label{eq: Bessel Equation}
z^2 \frac {d^2 w} {d z^2}  (z) + z \frac {d w} {d z} (z) + \lp z^2 - \nu^2 \rp  w(z) = 0.
\end{equation}
$J_{\nu} (z)$ is defined by the series    (see \cite[3.1 (8)]{Watson})
\begin{equation}\label{2def: series expansion of J}
J_{\nu} (z) = \sum_{n=0}^\infty \frac {(-)^n \lp \frac 1 2 z \rp^{\nu+2n } } {n! \Gamma (\nu + n + 1) }.
\end{equation}
When $\nu$ is not a negative integer, we have the bound (see for instance \cite[3.13 (1) or 3.31 (1, 2)]{Watson})
\begin{equation}\label{2eq: bound for J}
\left| J_{\nu} (z) \right| \lll_{\nu} \left|z^{ \nu}\right|, \hskip 10 pt |z| \leqslant 1.
\end{equation}

We have the following connection formulae  (see \cite[3.61 (1, 2)]{Watson})
\begin{align}
\label{2eq: J and H} 
& J_\nu (z) = \frac {H_\nu^{(1)} (z) + H_\nu^{(2)} (z)} 2, \hskip 30 pt 
J_{-\nu} (z) =  \frac {e^{\pi i \nu} H_\nu^{(1)} (z) + e^{-\pi i \nu} H_\nu^{(2)} (z) } 2.
\end{align}

We have the following asymptotics of $H^{(1)}_{\nu} (z)$ and $H^{(2)}_{\nu} (z)$ at infinity (see \cite[7.2 (1, 2)]{Watson}),
\begin{equation}\label{2eq: asymptotic H (1)}
H^{(1)}_{\nu} (z) = \lp \frac 2 {\pi z} \rp^{\frac 1 2} e^{ i \lp z - \frac 12 {\pi \nu}    - \frac 14 \pi    \rp }  \lp 1 +  \frac {   1   - 4 \nu^2} {8 i z}  + O \lp \frac 1 {|z|^{ 2}} \rp \rp,
\end{equation}
\begin{equation}\label{2eq: asymptotic H (2)}
H^{(2)}_{\nu} (z) = \lp \frac 2 {\pi z} \rp^{\frac 1 2} e^{ - i \lp z - \frac 12 {\pi \nu}    - \frac 1 4 \pi   \rp } \lp 1 - \frac {   1   - 4 \nu^2} {8 i z}  + O \lp \frac 1 {|z|^{ 2}} \rp \rp,
\end{equation}
of which \eqref{2eq: asymptotic H (1)} is valid when $z$ is such that $- \pi +  \delta \leqslant \arg z \leqslant 2 \pi -   \delta$,  and  \eqref{2eq: asymptotic H (2)} when $- 2 \pi +  \delta \leqslant \arg z \leqslant   \pi -  \delta$,   $  \delta  $ being any positive acute angle. Consequently, 
\begin{equation}\label{2eq: asymptotic J}
J_{\nu} \lp z \rp =  \lp \frac 2 {\pi z} \rp^{\frac 12}  \cos  \lp   z - \tfrac 1 2 \pi \nu - \tfrac 1 4 \pi \rp + O \lp  {|z|^{- \frac 3 2}} \rp,
\end{equation}
for $|\arg z| \leqslant \pi - \delta$.

According to \cite[3.63]{Watson}, $H^{(1)}_{\nu} (z)$ and $H^{(2)}_{\nu} (z)$  form a fundamental system of solutions of Bessel's equation.


\subsubsection{Basic Properties of  $I_{\nu} (z)$ and $K_{\nu}  (z) $} Let $I_{\nu}  (z) $ and $K_{\nu} (z)$ denote the modified Bessel function of the first and second kind of order $\nu$, which are defined by \cite[3.7 (2, 6)]{Watson},
\begin{align}
\label{2eq: I, K and J}  I_{\nu} (z) = e^{- \frac 1 2 \pi i \nu} J_{\nu} \big(e^{\frac 1 2 \pi i} z\big), \hskip 10 pt K_{\nu} (z) = \tfrac 1 2 \pi \shskip  \frac {I_{- \nu} (z) - I_{\nu} (z) } {\sin (\pi \nu)}. 
\end{align}
We have the following asymptotics of $I_{\nu}  (z) $ and $K_{\nu} (z)$ at infinity (\cite[7.23 (1, 2, 3)]{Watson}),
\begin{align}
\label{2eq: asymptotic I} & I_{\nu} (z) = \frac {e^z} {(2 \pi z)^{\frac 1 2}} \lp 1 + O\lp |z|\- \rp \rp, \\
\label{2eq: asymptotic K} & K_{\nu} (z) = \lp \frac {\pi} {2 z} \rp^{\frac 1 2} e^{- z} \lp 1 + O\lp |z|\- \rp \rp,
\end{align}
of which \eqref{2eq: asymptotic I} is valid when $z$ is such that $| \arg z | \leqslant \frac 1 2 \pi -   \delta$,  and  \eqref{2eq: asymptotic K} when $| \arg z | \leqslant  \frac 3 2 \pi -    \delta$.

In addition, we have the following recurrence formulae for $I_{\nu}  (z) $ and $K_{\nu} (z)$ (see \cite[3.71 (3)]{Watson} and \cite[9.6.29]{A-S}),
\begin{align}
\label{2eq: recurrence, I} z I_{\nu }' (z) + \nu I_{\nu } (z) & =   z I_{\nu - 1} (z), \\
\label{2eq: recurrence, I and K}
2^{n} I_{\nu}^{(n)} (z) = \sum_{r=0}^n C^{r}_{n} I_{\nu+n-2r} (z), \hskip 10 pt & (-2)^{n} K_{\nu}^{(n)} (z) = \sum_{r=0}^n C^{r}_{n} K_{\nu+n-2r} (z),
\end{align}
where $C^r_n$ is the binomial coefficient.

\subsubsection{Integral Formulae}

First, we shall need the following integral formula (\cite[8.6 (14)]{ET-II}),  for $ y > 0$, $|\arg c| < \frac 1 2 \pi$ and $\Re \nu > - 2$,
\begin{equation}\label{2eq: integral exp}
\int_0^\infty J_{\nu} ( x y) e^{- c x^2} x  d x = \frac {\Gamma \lp \frac 1 2   \nu + 1 \rp y^{\nu} } {2^{\nu+1}  \Gamma (\nu + 1) c^{\frac 1 2 \nu + 1} } 
M \lp \frac 1 2 \nu + 1; \nu+1; - \frac {y^2} {4 c} \rp,
\end{equation}
where $M (a; b; z)$ is Kummer's confluent hypergeometric function (see \S \ref{sec: Kummer}). 
Second, when $|\arg z| < \frac 1 2 \pi$, we have the integral representation of $K_{\nu} (z)$ (see \cite[6.22, (5, 7)]{Watson}),
\begin{equation}\label{2eq: integral repn of K}
K_{\nu} (z) = \frac 1 2 \int_0^\infty y^{\nu - 1} e^{- \frac 1 2 z \lp y + y\- \rp} d y .
\end{equation}
Furthermore, when the order $\nu = n$ is an integer, we have the integral representations of Bessel for $J_n (z)$ and $I_n (z)$ as follows (see \cite[2.2 (1)]{Watson}),
\begin{align}\label{2eq: integral repn of J}
&J_{n} (z) = (-)^n J_{- n} (z) = \frac { 1 } {2 \pi i^{\hskip 0.5 pt n}} \int_0^{2\pi} e^{- i n \phi + i z \cos \phi } d \phi , 
\\
\label{2eq: integral repn of I}
&\hskip 10 pt I_{n} (z) =   I_{- n} (z) = \frac { (-1)^n } {2 \pi} \int_0^{2\pi} e^{- i n \phi - z \cos \phi  } d \phi.
\end{align}

\subsection{Kummer's Confluent Hypergeometric Functions}\label{sec: Kummer}
When $b $ is not a nonpositive integer, Kummer's confluent hypergeometric Function $M (a; b; z)$ is defined by
\begin{align}
M (a; b; z) = {_1F_1} (a; b; z) = \frac {\Gamma (b)} {\Gamma (a)} \sum_{n=0}^\infty \frac {\Gamma (a + n)} {\Gamma (b+n) n!} z^n.
\end{align}
It is clear that
\begin{equation}\label{2eq: M(a; a)}
M (b; b; z) = e^{z}.
\end{equation}
According to \cite[13.2.1]{A-S}, when $ \Re b > \Re a > 0 $, we have
\begin{align}\label{2eq: integral repn of M(a; b)}
\frac {\Gamma (b-a) \Gamma (a)} {\Gamma (b)} M (a; b; z) = \int_0^1 e^{z \varv} \varv^{a-1} (1-\varv)^{b-a-1} d \varv.
\end{align}

\subsection{\texorpdfstring{Preliminaries on the Bessel Function $\bfJ_{\mu,\shskip  m} (z)$}{Preliminaries on the Bessel Function $J_{\mu,\shskip  m} (z)$}}



\subsubsection{}

Replacing $d/dz$ by $\partial / \partial z$, we denote by $\nabla_{\nu}$ the differential operator that occurs in \eqref{eq: Bessel Equation}, namely,
\begin{equation}\label{2eq: nabla}
\nabla_{\nu} = z^2 \frac {\partial^2 } {\partial z^2}  + z \frac {\partial  } {\partial z} +   z^2 - \nu^2    .
\end{equation}
Its conjugation will be  denoted by $\overline \nabla_{\nu}$,
\begin{equation}\label{2eq: nabla bar}
\overline \nabla_{\nu} = \overline z^2 \frac {\partial^2 } {\partial \overline z^2}  + \overline z \frac {\partial  } {\partial \overline z} +  \overline z^2 - \nu^2    .
\end{equation}
From the definition of $\bfJ_{ \mu,\shskip   m} (z) $ as in (\ref{0def: J mu m (z)}, \ref{0eq: defn of Bessel}), we infer that 
\begin{align}\label{2eq: nabla J = 0}
\nabla_{2 \mu + \frac 1 2 m} \lp \bfJ_{\mu,\shskip  m} \big(   z^2 /16 \pi^2 \big) \rp = 0, \hskip 10 pt \overline \nabla_{2 \mu - \frac 1 2 m} \lp \bfJ_{\mu,\shskip  m} \big(   z^2 /16 \pi^2 \big) \rp = 0.
\end{align} 

It follows from \eqref{2eq: bound for J} that if $\mu$ is generic, that is $4 \mu \notin 2 \BZ + m$, then
\begin{align}\label{2eq: bound for J mu m}
\left|\bfJ_{ \mu,\shskip   m} (z) \right| \lll_{\, \mu,\shskip   m} \left| |z|^{- 2 \mu} \right| + \left| |z|^{ 2 \mu} \right|, \hskip 10 pt |z| \leqslant 1.
\end{align} 
Some calculations by the formulae  of $ \left.( \partial J_{\nu}  (z) /\partial \nu ) \right|_{\nu = \pm n}$ in \cite[\S 3.52 (1, 2)]{Watson}, with nonnegative integer  $n$, would imply that in the generic case when $4 \mu \in 2 \BZ + m$ we have
\begin{align}\label{2eq: bound for J mu m, 2}
\left|\bfJ_{ \mu,\shskip   m} (z) \right| \lll_{\, \mu,\shskip   m}   |z|^{- 2 |\mu|} \log (2/|z|), \hskip 10 pt |z| \leqslant 1.
\end{align}

In view of the connection formulae in \eqref{2eq: J and H}, we have another expression of $ \bfJ_{ \mu,\shskip   m} (z)$ in terms of Hankel functions,
\begin{equation}\label{2eq: J = H1 + H2}
\bfJ_{ \mu,\shskip   m} (z) = \pi^2 i  \lp e^{2 \pi i \mu} H^{(1)}_{\mu,\shskip   m} \lp 4 \pi \sqrt z \rp + (-)^{m+1} e^{- 2 \pi i \mu} H^{(2)}_{\mu,\shskip   m} \lp 4 \pi \sqrt z \rp \rp,
\end{equation}
with the definition
\begin{equation}\label{7def: H (1, 2) mu m (z), n=2, C}
H^{(1,\shskip  2)}_{\mu,\shskip   m} (z) = H^{(1,\shskip  2)}_{2 \mu + \frac 1 2 m} \lp   z \rp  H^{(1,\shskip  2)}_{2 \mu - \frac 1 2 m} \lp  { \overline z} \rp.
\end{equation}
It follows from (\ref{2eq: asymptotic H (1)}, \ref{2eq: asymptotic H (2)}) that $ \bfJ_{ \mu,\shskip   m} \lp z   \rp $ admits the following asymptotic at infinity,
\begin{equation}\label{2eq: asymptotic of J mu m}
\begin{split}
\bfJ_{ \mu,\shskip   m} (z) =  \sum_{\pm} \frac {(\pm 1)^m} {2 \sqrt {|z|}} 
e \lp \pm 2 \big(\sqrt z + \sqrt {\overline z} \big) \rp \Bigg( 1 
  \pm \frac {1 - 4 \lp \mu + \frac 1 2 m \rp^2 } {8 i \sqrt z}  
  \pm \frac {1 - 4 \lp \mu - \frac 1 2 m \rp^2 } {8 i \sqrt {\overline z} } \Bigg) \\
 + O \lp |z|^{-\frac 3 2} \rp.
\end{split}
\end{equation}
In particular, combining \eqref{2eq: bound for J mu m} and \eqref{2eq: asymptotic of J mu m}, if we let  $\rho = |\Re \mu |$,  then 
\begin{equation}\label{2eq: bounds for J}
\bfJ_{ \mu,\shskip   m} \lp z \rp \lll_{\mu, \shskip m} \left\{\begin{split}
& 1 / \left| z  \right|^{2 \rho}, \hskip 11 pt \text{ if } |z| \leqslant 1, \\
& 1 /{\textstyle \sqrt {|z|}}, \hskip 10 pt \text{ if } |z| > 1.
\end{split}\right.
\end{equation} 
The first estimate in \eqref{2eq: bounds for J} is for generic $\mu$, but it remains valid in general if we let $ \rho > |\Re \mu | $ (see \eqref{2eq: bound for J mu m, 2}). 

\begin{lem}\label{lem: Bessel equation} 
	Let $f (z)$ be a solution of the following two differential equations,
	\begin{equation*}
	\nabla_{2 \mu + \frac 1 2 m} w = 0, \hskip 10 pt \overline \nabla_{2 \mu - \frac 1 2 m} w = 0,
	\end{equation*} 
	with differential operators $\nabla_{2 \mu + \frac 1 2 m} $ and $\overline \nabla_{2 \mu - \frac 1 2 m}$ defined as in \eqref{2eq: nabla} and \eqref{2eq: nabla bar}.
	Suppose further that $f (4 \pi z)$ admits the same asymptotic of $\bfJ_{\mu,\shskip  m} \lp z^2   \rp$, that is,
	\begin{align*}
	f (4 \pi z) \sim \frac {1} {2    {|z|}} 
	e \lp 2 (  z +   {\overline z} ) \rp +  \frac {(-1)^m} {2    {|z|}} 
	e \lp - 2 (  z +   {\overline z} ) \rp, \hskip 10 pt |z| \ra \infty.
	\end{align*}
	Then $ f (4 \pi z) = \bfJ_{\mu,\shskip  m}  ( z^2 )   $.
\end{lem}

\begin{proof}
	From the theory of differential equations,  $f (z)$ may be uniquely written as a linear combination of $H^{(k)}_{2 \mu + \frac 1 2 m} (z) H^{(l)}_{2 \mu - \frac 1 2 m} (\overline z)$, with $k, l = 1, 2$, namely, 
	$$ f(z) = \underset{k,\shskip  l = 1,\shskip  2}{\sum \sum} \ c_{k l}\shskip  H^{(k)}_{2 \mu + \frac 1 2 m} (z) H^{(l)}_{2 \mu - \frac 1 2 m} (\overline z).$$
	Letting $z = - 4 \pi i x $, with $x$ positive, we infer from (\ref{2eq: asymptotic H (1)}, \ref{2eq: asymptotic H (2)}) that if $c_{12} \neq 0$ then
	\begin{equation*}
	f (-4 \pi i x) \sim \frac {c_{12}} {2 \pi i^m x} \exp (8 \pi x), \hskip 10 pt x \ra \infty.
	\end{equation*}
	However, the asymptotic of $ \bfJ_{\mu,\shskip  m} \lp - x^2 \rp $ is $\lp 1 + (-1)^m \rp / 2x$, so we must have  $c_{12} = 0$ in order for  $f (4 \pi i x) $ and $ \bfJ_{\mu,\shskip  m} \lp - x^2 \rp $ to have the same asymptotic. Similarly, $c_{21} = 0$.
	Choosing $z = 2k \pi , (2k-1) \pi$, with $k$ positive integer, and letting $k \ra \infty$, it follows that $c_{11} = \pi^2 i    e^{2 \pi i \mu}$ and $ c_{22} =  \pi^2 i  (-1)^{m+1}  e^{- 2 \pi i \mu}$. Hence, by \eqref{2eq: J = H1 + H2}, we must have  $ f (4 \pi z) = \bfJ_{\mu,\shskip  m}   (z^2)    $.
\end{proof}

\subsubsection{An Integral Representation of $\bfJ_{ \mu,\shskip   m}$}
In the polar coordinates, we have the following integral representation of $ \bfJ_{ \mu,\shskip   m} \lp x e^{i \phi} \rp $ (see \cite[Corollary 6.17]{Qi-Bessel} and \cite[Theorem 12.1]{B-Mo}),
\begin{equation}\label{2eq: integral repn of J mu m}
\bfJ_{\mu,\shskip  m} \lp x e^{  i \phi} \rp = 4 \pi i^m \int_0^\infty   y^{4 \mu - 1}  E \big(y e^{\frac 1 2 i \phi} \big)^{-m} J_m \lp 4 \pi \sqrt x Y \big(y e^{\frac 1 2 i \phi} \big) \rp d y,
\end{equation}
with
\begin{align*}
& Y (z) = \left| z + z^{-1}  \right|, 
\hskip 10 pt E (z) = \lp z + z\- \rp /\left| z + z^{-1}  \right|.
\end{align*}
The integral on the right of \eqref{2eq: integral repn of J mu m} is absolutely convergent if $|\Re \mu | < \frac 1 8$.



\subsection{Stationary Phase Integrals}

The lemma below is a special case of \cite[Theorem 7.7.5]{Hormander}.

\begin{lem}\label{lem: stationary phase}
	Let $K \subset \BC \smallsetminus\{0\}$ be a compact set, $X$ an open neighbourhood of $K$. In the polar coordinates, if $u(x, \phi)  = u \lp x e^{i \phi} \rp \in C^{2 }_0 (K)$, $f(x, \phi) =  f \lp x e^{i \phi} \rp  \in C^{4} (X)$ and $ f  $ is a real valued function on $X$, $  f (x_0, \phi_0) = 0$, $ f' (x_0, \phi_0) = 0$, $\det f'' (x_0, \phi_0) \neq 0$ and $ f'  \neq 0$ in $K \smallsetminus \{ (x_0, \phi_0) \}$, then for $y > 0$
	\begin{equation*}
	\begin{split}
	  \sideset{}{_K}{\iint}  u(x , \phi ) e \lp y f (x , \phi ) \rp d x d \phi = \frac {   u (x_0, \phi_0) e \lp y f(x_0, \phi_0) + \frac 1 4 \rp }  { y \sqrt {\det f'' (x_0, \phi_0) }  }   + O \lp \frac 1 {y^2}  \rp.
	\end{split}
	\end{equation*}
	Here the implied constant depends only on  $f$, $u$ and $K$. 
\end{lem}

\section{Combinatorial Lemmas and Recurrence Formulae for Classical Bessel Functions}\label{sec: combinatorial}


Let $C_n^r$ denote binomial coefficients. By convention, we let $C_n^r = 0$ if either $r < 0$ or $0\leqslant n < r$. Throughout this section, unless otherwise specified, we   assume that the numbers $k, l, n, r, s ...$ are nonnegative integers.

\subsection{A Combinatorial Inversion Formula}

First, we have the following combinatorial inversion formula. It will be applied to  prove   the inversion of the recurrence formulae in \eqref{2eq: recurrence, I and K} in  Lemma \ref{lem: recurrence, case l = 0} and its generalization in Lemma \ref{lem: recurrence, general case}.
\begin{lem}\label{lem: inversion formula}
	For $  n \geqslant 2 r $, we define $D_n^r = (-)^r \big( C_{n-r}^r + C_{n-r-1}^{r-1} \big)$. Suppose $\left\{ f_n \right\}$ and $\left\{ g_n \right\}$ are two sequences of complex numbers such that
	\begin{align}\label{2eq: rel 1}
	f_n = \sum_{r = 0}^{\left\lfloor n/2 \right \rfloor} C_n^r \shskip  g_{n-2r},
	\end{align}
	then
	\begin{align}\label{2eq: rel 2}
	g_n = \sum_{r = 0}^{\left\lfloor n/2 \right \rfloor} D_n^r \shskip  f_{n-2r}.
	\end{align}
	Conversely, if $\left\{ g_n \right\}$ is constructed from $\left\{ f_n \right\}$ by \eqref{2eq: rel 2}, then \eqref{2eq: rel 1} holds.
\end{lem}	

\begin{proof}
	The first statement may be easily proven by induction once the following identity is verified
	\begin{equation}\label{2eq: sum C D = delta}
	\sum_{r=0}^s C_{n}^{r} \shskip  D_{n-2r}^{ s-r} = \delta_{s, \shskip 0},  
	\end{equation}
	for $ n \geqslant 2 s $, where $\delta_{s, \shskip 0}$ is the Kronecker symbol that detects $s = 0$. The second statement is simply a matter of uniqueness.
	
	We first prove 
	\begin{align}\label{2eq: sum of C D, 1}
	\sum_{r=0}^s (-)^{s-r} C_{n}^{r} \shskip  C_{n-s-r}^{s-r} = C_{2s-1}^s.
	\end{align}
	For this, we consider the identity
	\begin{align*}
	(1-X)^{n-s} \lp 1 + \frac X {1 - X} \rp^n = \frac 1 {(1 - X)^s}.
	\end{align*}
	The left hand side expands as
	\begin{align*}
	\sum_{r }  C_n^r\shskip  X^r (1-X)^{n-s-r} = \underset{r, \shskip  t}{\sum \sum}  (-)^t C_n^r C_{n-s-r}^t \shskip   X^{t+r},
	\end{align*}
	whereas the right hand side  expands as
	\begin{align*}
	\sum_{p} (-)^p C_{-s}^p \shskip  X^{p} = \sum_{p} C_{s+p-1}^p \shskip  X^{p}.
	\end{align*}
	Then the identity follows immediately from comparing the coefficients of $X^s$. Similarly, on comparing the coefficients of $X^{s-1}$ in the identity
	\begin{align*}
	(1-X)^{n-s-1} \lp 1 + \frac X {1 - X} \rp^n =  \frac {1} {(1-X)^{s+1}},
	\end{align*}
	we find  that 
	\begin{align}\label{2eq: sum of C D, 2}
	\sum_{r=0}^{s-1} (-)^{s-r} C_{n}^{r} \shskip  C_{n-s-r-1}^{s-r-1} = - C_{2s-1}^{s-1}.
	\end{align}
	Summing \eqref{2eq: sum of C D, 1} and \eqref{2eq: sum of C D, 2} yields \eqref{2eq: sum C D = delta}.
\end{proof}

\subsection{A Combinatorial Identity}

The following identity will be crucial for generalizing the inversion of the recurrence formulae in \eqref{2eq: recurrence, I and K}. See Lemma \ref{lem: recurrence, general case}.

\begin{lem}\label{lem: combinatorics 1}
	For $0 \leqslant r \leqslant n - l $, we define $ B_{l,\shskip  n}^r = C_{l+r }^r C_{n-r}^{n-l-r} - C_{l+r -1}^{r-1 } C_{n-r-1}^{n-l-r-1}$. Then
	\begin{align*}
	\sum_{r = 0}^s C_n^r B_{l,\shskip  n-2r}^{s-r} = C_{n}^l C_{n-l}^s,  
	\end{align*}
	for $0\leqslant s \leqslant n - l$.
\end{lem}

\begin{proof}
	Consider 
	\begin{align*}
	P_{l,\shskip  n} (X, Y) = \frac {(1-XY) (1+XY)^n} {\lp (1-X) (1-Y) \rp^{l+1}}.
	\end{align*}
	First, we expand $P_{l,\shskip  n} (X, Y)$ as below,
	\begin{align*}
	P_{l,\shskip  n} (X, Y)    = \underset{r,\shskip  p,\shskip  q}{\sum \sum \sum }   C_{n}^r  C_{l+p}^p C_{l+q}^q \lp X^{r+p} Y^{r+q} - X^{r+p+1} Y^{r+q+1} \rp,
	\end{align*}
	Hence the left hand side of the identity is exactly the coefficient  of $X^{s} Y^{n-l-s}$ in $P_{l,\shskip  n} (X, Y) $. Second, we write $P_{l,\shskip  n} (X, Y)$ in another way,
	\begin{align*}
	P_{l,\shskip  n} (X, Y) & = \frac {(1-XY) (1+XY)^{n-l-1}} {\lp 1 - (X+Y) / (1+XY) \rp^{l+1}} \\
	& = \sum_{ t } C_{l+t}^l  (X+Y)^{t} (1-XY) (1+ XY)^{n - l - t - 1}.
	\end{align*}
	Thus the degree-$(n-l)$ homogeneous part of $P_{l,\shskip  n} (X, Y)$ is equal to
	\begin{align*}
	C_{n}^l (X+Y)^{n-l} + \sum_{\sstyle t = 0  \atop \sstyle t\shskip  \equiv\shskip  n-l (\mod 2) }^{n-l-2 } & \lp  C_{n-l-t-1}^{\frac 1 2 \lp n-l-t \rp} - C_{n-l-t-1}^{\frac 1 2 \lp n-l-t \rp - 1} \rp (X+Y)^t (XY)^{\frac 1 2 \lp n-l-t \rp} \\
	& \hskip 35 pt =    C_{n}^l (X+Y)^{n-l} = \sum_{ s = 0}^{n-l} C_{n}^l C_{n-l}^s X^{s} Y^{n-l-s}.
	\end{align*}
	The proof is complete by comparing the coefficients of $X^{s} Y^{n-l-s}$.
\end{proof}
We   note that $B_{l,\shskip  n}^r = B_{l,\shskip  n}^{n-l-r}$. Moreover, if $(l, r ) \neq (0, 0), (0, n)$ it would be preferable to write $ B_{l,\shskip  n}^r = C_{l+r }^l C_{n-r}^{l} - C_{l+r -1}^{l} C_{n-r-1}^{l}$. 


\subsection{Recurrence Formulae for Classical Bessel Functions}

\begin{lem}\label{lem: recurrence, case l = 0}
	Let notations be as above. Define   $I_{\nu,\shskip  0} (z) = I_{\nu} (z) $, $K_{\nu,\shskip  0} (z) = K_{\nu} (z)$, and, for $n \geqslant 1$,
	$I_{\nu,\shskip  n} (z) = I_{\nu - n} (z) + I_{\nu + n} (z)$, $ K_{\nu,\shskip  n} (z) = K_{\nu - n} (z) + K_{\nu + n} (z)$.  Then
	\begin{align*} 
	\sum_{r = 0}^{\left\lfloor   n/2 \right\rfloor}   D_n^r \cdot 2^{n-2r} I_{\nu}^{(n-2r)} (z) = I_{\nu,\shskip  n} (z),   \hskip 10 pt \sum_{r = 0}^{\left\lfloor   n/2 \right\rfloor}   D_n^r \cdot (-2)^{n-2r} K_{\nu}^{(n-2r)} (z) = K_{\nu,\shskip  n} (z).
	\end{align*} 
\end{lem}

\begin{proof}
	Note that we may reformulate \eqref{2eq: recurrence, I and K} as
	\begin{align}
	\label{3eq: recurrence 1, l=0} 2^{n} I_{\nu}^{(n)} (z) = \sum_{r = 0}^{\left\lfloor   n/2 \right\rfloor} C^{r}_{n} I_{\nu, \shskip  n - 2r} (z), & \hskip 10 pt (-2)^{n} K_{\nu}^{(n)} (z) =   \sum_{r = 0}^{\left\lfloor   n/2 \right\rfloor} C^{r}_{n} K_{\nu, \shskip  n - 2r} (z).
	\end{align}
	This lemma is therefore a directly consequence of Lemma \ref{lem: inversion formula}.	
\end{proof}

With the help of Lemma \ref{lem: combinatorics 1}, we generalize the first formula in Lemma \ref{lem: recurrence, case l = 0} as follows. 

\begin{lem}\label{lem: recurrence, general case}
	Let notations be as above. Suppose that $ l \leqslant n $. We have
	\begin{align*}
	\sum_{r = 0}^{\left\lfloor   (n-l)/2 \right\rfloor}  D_{n}^r \cdot 2^{n-2r} C_{n-2r}^l \shskip 
	I_{ \nu }^{(n - l - 2r )} (z)  =  2^l \sum_{s=0}^{n-l} B_{l,\shskip  n}^s I_{\nu - n+l + 2 s} (z).
	\end{align*}
	
\end{lem}

\begin{proof}
	In the notations of Lemma \ref{lem: inversion formula}, we let 
	\begin{equation*}
	f_{\shskip  l,\shskip  n} = \left\{  \begin{split}
	& 2^{n } C_{n }^l \shskip 
	I_{ \nu }^{(n - l  )} (z), \hskip 10 pt \text{ if } n \geqslant l, \\
	& 0, \hskip 61 pt \text{ if }   n < l,
	\end{split} \right.
	\end{equation*}
	and we only need to prove 
	\begin{equation*}
	g_{l,\shskip  n} = \left\{  \begin{split}
	& 2^l \sum_{s=0}^{\left\lfloor (n-l)/2 \right\rfloor} B_{l,\shskip  n}^s I_{\nu,\shskip  n-l - 2 s} (z), \hskip 10 pt \text{ if } n \geqslant l, \\
	& 0, \hskip 107 pt \text{ if }   n < l.
	\end{split} \right.  
	\end{equation*} 
	By Lemma \ref{lem: combinatorics 1} and \eqref{3eq: recurrence 1, l=0},
	\begin{align*}
	& \hskip 11 pt \sum_{r = 0}^{\left\lfloor (n-l) /2 \right \rfloor} C_n^r \cdot 2^l \sum_{s=0}^{\left\lfloor (n-l - 2 r)/2 \right\rfloor} B_{l,\shskip  n-2r}^s \shskip  I_{\nu,\shskip  n- l - 2 r - 2 s} (z) \\
	& =  2^l \sum_{s=0}^{\left\lfloor (n-l)/2 \right\rfloor}  I_{\nu,\shskip  n-l - 2 s} (z) \sum_{r = s}^{n-l} C_n^r  B_{l,\shskip  n-2r}^{s-r} \\
	& = 2^l C_{n}^l \sum_{s=0}^{\left\lfloor (n-l)/2 \right\rfloor}   C_{n-l}^s \shskip  I_{\nu,\shskip  n-l - 2 s} (z) \\
	& = 2^{n } C_{n }^l \shskip 
	I_{ \nu }^{(n - l  )} (z).
	\end{align*}
	This verifies the nontrivial case of the identity \eqref{2eq: rel 1} in Lemma \ref{lem: inversion formula} for $f_{l,\shskip  n}$ and $g_{l,\shskip  n}$. Therefore, our proof is done with an application of Lemma \ref{lem: inversion formula}.
\end{proof}

A similar recurrence formula holds for $K_{\nu} (z)$, or other types of Bessel functions, but only the formula for $I_{\nu} (z)$ will be needed in the sequel.

\subsection{Additional Combinatorial Identities} In the following, we collect some  combinatorial results that will be used in the proof of  Theorem \ref{thm: integral of J} (or, indeed Proposition \ref{prop: recursive S(a, v)}).

\begin{lem}\label{lem: combinatorics 2}
	Let $B_{l,\shskip  n}^r$ be defined as in Lemma {\rm \ref{lem: combinatorics 1}}. For $0 \leqslant r \leqslant k-l$, we define $A_{k,\shskip  l}^r =    C_{l+r-1}^{r-1} C_{2k-l-1}^{k-l-r-1} +  C_{l+r }^r C_{2k-l-1}^{k-l-r } $. Then
	\begin{align*}
	\sum_{n=l+r}^k (-)^n C_{2k}^{k-n} B_{l,\shskip  n}^r = (-)^{l+r} A_{k,\shskip  l}^r.
	\end{align*}
\end{lem}
\begin{proof}
	In view of the definition of $B_{l,\shskip  n}^r$, it suffices to prove the identity
	\begin{align*}
	\sum_{n=l+r}^k (-)^n C_{2k}^{k-n} C_{n-r}^{n-l-r} = (-)^{l+r} C_{2k-l-1}^{k-l-r}.
	\end{align*}
	This follows easily from examining the coefficients of  $X^{k-l-r }$ in the identity
	\begin{align*}
	\frac {( (1+X) - X)^{2k} } {(1 + X)^{k+r}} = \frac { 1 } {(1+X)^{k+r}}.
	\end{align*}	
\end{proof}
By our conventions on binomial coefficients, $A_{k,\shskip  l}^{r}$ is well defined for all values of $k, l$ and $r$, but it vanishes except for  $0 \leqslant r \leqslant k - l $.
Moreover, when  $(l, r) \neq (0, 0)$, we shall always write $A_{k,\shskip  l}^r =    C_{l+r-1}^{l} C_{2k-l-1}^{k-l-r-1} +  C_{l+r }^l C_{2k-l-1}^{k-l-r } $.

The following lemma includes some formulae for $A_{k,\shskip  l}^r$ and is designed for the concluding part of the proof of Theorem \ref{thm: integral of J} (Proposition \ref{prop: recursive S(a, v)}). These seemingly complicated formulae may be verified straightforwardly with the help of Mathematica. 

\begin{lem}\label{lem: combinatorics 3}
	Suppose that $k \geqslant 1$.
	
	{\rm (1).} We have
	\begin{equation}\label{3eq: Ak0}
	\begin{split}
	& A_{k+1, \shskip 0}^0 = 2 A_{k , \shskip 0}^0 +  A_{k , \shskip 0}^1, \hskip 10 pt A_{k+1, \shskip 0}^1 = 2 A_{k , \shskip 0}^0 + 2 A_{k , \shskip 0}^1 +    A_{k , \shskip 0}^2, \\
	& A_{k+1, \shskip 0}^r =   A_{k , \shskip 0}^{r - 1} + 2 A_{k , \shskip 0}^{r } + A_{k , \shskip 0}^{r+1}  ,  \hskip 10 pt r \geqslant 2,
	\end{split}
	\end{equation}
	
	{\rm (2).} We have
	\begin{align}\label{3eq: Ak1}
	A_{k+1, \shskip 1}^0 = (k+1) \big( 2 A_{k , \shskip 0}^0 - A_{k , \shskip 0}^{1} \big), \hskip 10 pt A_{k+1, \shskip 1}^r = (k+1) \big(  A_{k , \shskip 0}^r - A_{k , \shskip 0}^{r+1} \big) ,  \hskip 10 pt r \geqslant 1.
	\end{align}
	\begin{equation}\label{3eq: Ak1, 2}
	\begin{split}
	&A_{k+1, \shskip 1}^0 - \lp 2 A_{k , \shskip 0}^0 - A_{k , \shskip 0}^1 + A_{k , \shskip 1}^0 + A_{k , \shskip  1}^1 \rp = 0, \\
	& 	A_{k+1, \shskip 1}^r   -   \big(       A_{k, \shskip 0}^r -   A_{k, \shskip 0}^{r+1}  + A_{k, \shskip 1}^{r-1} + 2 A_{k, \shskip 1}^{r } + A_{k, \shskip 1}^{r+1} \big)  =  0, \hskip 10 pt r \geqslant 1.
	\end{split} 
	\end{equation}
	
	{\rm (3).} For  $l \geqslant 2$ and $k-l+1 \geqslant r \geqslant - 1$, we have
	\begin{equation}\label{3eq: Akl, 0}
	\begin{split}
	(k-1) (l-1) l   A^{r }_{k+1,\shskip  l} \ - \ &  (k-1)  (k+1)  (k-l+2) A^{r+1 }_{k ,\shskip  l-2} \\
	= - \ & (k+1) (2k-l) (2k-l+1) A^{r+1 }_{k-1 ,\shskip  l-2},
	\end{split}
	\end{equation}
	\begin{align}\label{3eq: Akl, 1}
	l A^{r }_{k+1,\shskip  l} - (k - l+2) \big(    A_{k,\shskip  l-2}^{r+1} + A_{k,\shskip  l-1}^{r } -  A_{k,\shskip  l-1}^{r+1} \big) = - (2k-l) A_{k-1,\shskip  l-2}^{r+1},
	\end{align}
	\begin{equation}\label{3eq: Akl, 2}
	\begin{split}
	A_{k+1,\shskip  l}^r   -   \big(       A_{k,\shskip  l-1}^r -   A_{k,\shskip  l-1}^{r+1}  + A_{k,\shskip  l}^{r-1} + 2 A_{k,\shskip  l}^{r } + A_{k,\shskip  l}^{r+1} \big)  = A_{k-1,\shskip  l-2}^{r+1} .
	\end{split} 
	\end{equation}
\end{lem}

\begin{proof} 
	\
	
	(1).  Note that
	\begin{align*}
	A_{k, \shskip 0}^0 = C_{2k-1}^{k } = \tfrac 1 2 C_{2k }^{k }, \hskip 10 pt A_{k, \shskip 0}^r = C_{2k-1}^{k-r} + C_{2k-1}^{k-r-1} = C_{2k}^{k-r},  \hskip 10 pt k, r \geqslant 1.
	\end{align*} 
	By Pascal's rule, 
	\begin{align*}
	C_{2k+2}^{k-r+1} = C_{2k }^{k-r-1} + 2 C_{2k }^{k-r } + C_{2k }^{k-r+1},
	\end{align*}
	which implies  \eqref{3eq: Ak0}.
	
	(2). We have
	\begin{align*}
	A_{k, \shskip 1}^r = r \shskip  C_{2k-2}^{k-r-2} + (r+1) C_{2k-2}^{k-r-1} .
	\end{align*}
	The simple identity below yields \eqref{3eq: Ak1},
	\begin{align*}
	r \shskip  C_{2k }^{k-r-1} + (r+1) C_{2k }^{k-r }   = 
	(k+1) \lp C_{2k }^{k-r } - C_{2k  }^{k - r - 1} \rp.
	\end{align*}
	Again, \eqref{3eq: Ak1, 2} follows from Pascal's rule.
	
	(3). Let us assume $r \geqslant 0$. The degenerated case $r = -1$ is much simpler. 
	
	The identity \eqref{3eq: Akl, 0} is equivalent to the vanishing of 
	\begin{align*}
	& \hskip 12 pt (k - 1) (r + 1) (l + r -  1) \big( r (k - l - r + 1) + (l + r) (k + r + 1) \big) \\
	&- (k - 
	1) (k + 1)  (k - l +  2) \big( (r + 1) (k - l - r + 1)  + (l + r - 1) (k + r + 1) \big) \\
	& + (k + 
	1) (k + r + 1) (k - l - r + 
	1)  \big( (r + 1) (k - l - r )  + (l + r - 1) (k + r) \big),
	\end{align*}
	which may be checked directly. As for \eqref{3eq: Akl, 1},  we verify by Pascal's rule that
	\begin{align}\label{3eq: aux, 1}
	A_{k,\shskip  l-2}^{r+1} + A_{k,\shskip  l-1}^{r } -  A_{k,\shskip  l-1}^{r+1} = 
	C_{l+r }^{l-1} C_{2k-l}^{k-l-r+1} - C_{l+r -2}^{l-1} C_{2k-l}^{k-l-r -1},
	\end{align}
	\begin{align}\label{3eq: aux, 2}
	A^{r }_{k+1,\shskip  l} = C_{l+r-1}^l C_{2k-l}^{k-l-r -1} + \lp C_{l+r-1}^l + C_{l+r }^l \rp C_{2k-l}^{k-l-r} + C_{l+r }^l C_{2k-l}^{k-l-r+1}.
	\end{align} 
	Then,
	\begin{align*}
	& \hskip 13 pt l A^{r }_{k+1,\shskip  l} - (k - l+2) \big(    A_{k,\shskip  l-2}^{r+1} + A_{k,\shskip  l-1}^{r } -  A_{k,\shskip  l-1}^{r+1} \big) \\
	& = \big(  l C_{l+r-1}^l + (k-l+2) C_{l+r-2}^{l-1} \big) C_{2k-l}^{k-l-r-1} + l \big(  C_{l+r-1}^l + C_{l+r }^l \big) C_{2k-l}^{k-l-r} \\
	& \hskip 12 pt + \big(  l C_{l+r}^l - (k-l+2) C_{l+r}^{l-1} \big) C_{2k-l}^{k-l-r+1} \\
	& =  \hskip - 1 pt (k + r+1) C_{l+r-2}^{l-1}   C_{2k-l}^{k-l-r-1} \hskip -2 pt + \hskip - 1 pt l \big( C_{l+r-1}^l + C_{l+r }^l \big) C_{2k-l}^{k-l-r} \hskip -2 pt -  \hskip -1 pt (k-l-r+1) C_{l+r}^{l-1} C_{2k-l}^{k-l-r+1} \\
	& = \big( (k-l-r) C_{l+r-2}^{l-1} + l \big(  C_{l+r-1}^l + C_{l+r }^l \big) - (k+r) C_{l+r}^{l-1} \big) C_{2k-l}^{k-l-r } \\
	& = - (k-1) \big(  C_{l+r-2}^{l-2} + C_{l+r-1}^{l-2} \big) C_{2k-l}^{k-l-r }.
	\end{align*}  
	Therefore, we are left to show
	\begin{align*}
	(k-1) \big(  C_{l+r-2}^{l-2} + C_{l+r-1}^{l-2} \big) C_{2k-l}^{k-l-r } = (2k-l) \big(  C_{l+r-2}^{l-2} C_{2k-l-1}^{k-l-r -1} + C_{l+r-1}^{l-2} C_{2k-l-1}^{k-l-r } \big).
	\end{align*}
	It is reduced to the following identity, which may be verified directly,
	\begin{align*}
	(k-1) \big( (r+1) + (l+r-1) \big) = (r+1) (k-l-r) + (l+r-1) (k+r).
	\end{align*}
	Finally, we verify \eqref{3eq: Akl, 2}. Similar to \eqref{3eq: aux, 1} and \eqref{3eq: aux, 2}, we have
	\begin{equation}
	\begin{split}
	A_{k ,\shskip  l-1}^{r} - A_{k ,\shskip  l-1}^{r+1} + A_{k ,\shskip  l}^{r-1} - 2 A_{k ,\shskip  l}^{r} + A_{k ,\shskip  l}^{r+1} = & \  C_{l+r -1}^{l } C_{2k-l-1}^{k-l-r -2}  - C_{l+r -2}^{l } C_{2k-l-1}^{k-l-r -1} \\
	& - C_{l+r+1 }^{l } C_{2k-l-1}^{k-l-r } + C_{l+r }^{l } C_{2k-l-1}^{k-l-r+1},
	\end{split}
	\end{equation}
	\begin{equation}
	\begin{split}
	A_{k+1 ,\shskip  l }^{r} - 4 	A_{k  ,\shskip  l }^{r} =   C_{l+r-1}^l C_{2k-l-1}^{k-l-r -2} + & \lp   C_{l+r }^l \hskip -2 pt - 2 C_{l+r-1}^l \rp C_{2k-l-1}^{k-l-r-1}   \\
	+  & \lp C_{l+r-1}^l \hskip -2 pt - 2 C_{l+r }^l \rp C_{2k-l-1}^{k-l-r} + C_{l+r }^l C_{2k-l-1}^{k-l-r+1}  .
	\end{split}
	\end{equation}
	Hence 
	\begin{align*}
	& \hskip 14 pt A_{k+1,\shskip  l}^r   -   \big(       A_{k,\shskip  l-1}^r -   A_{k,\shskip  l-1}^{r+1}  + A_{k,\shskip  l}^{r-1} + 2 A_{k,\shskip  l}^{r } + A_{k,\shskip  l}^{r+1} \big) \\
	& =  \lp C_{l+r -2}^{l } - 2 C_{l+r-1}^l + C_{l+r }^l \rp C_{2k-l-1}^{k-l-r-1} + \lp C_{l+r-1}^l - 2 C_{l+r }^l + C_{l+r+1 }^{l } \rp C_{2k-l-1}^{k-l-r} \\
	& = C_{l+r -2}^{l-2}  C_{2k-l-1}^{k-l-r-1} + C_{l+r -1}^{l-2} C_{2k-l-1}^{k-l-r} \\
	& = A_{k-1,\shskip  l-2}^{r+1}.
	\end{align*}
\end{proof}

\section{Proof of Theorem \ref{thm: integral of J}}

For brevity, we put $m = 2 k$. Without loss of generality, we assume that $ k \geqslant 0$, $|\Re \mu | < \frac 1 8$ and $\mu \neq 0$. The admissible range of  $\mu$ for \eqref{0eq: int of J = IK} in Theorem \ref{thm: integral of J} may be extended to $ |\Re \mu | < \frac 1 2 $ by the principal of analytic continuation.
	
\subsection{First Reductions}\label{sec: reduction}

The first step is to reduce the integral formula to a combinatorial identity which involves the derivatives of functions of the form $I_n (a \varv) \varv^{k} (1-\varv)^{k - 1}$ so that the combinatorial theory for  the derivatives of the $I$-Bessel function developed in \S \ref{sec: combinatorial} will then come in to play. 

We insert the integral representation of $\bfJ_{\mu,\shskip  2 k} \lp x e^{  i \phi} \rp$ in \eqref{2eq: integral repn of J mu m} and change the order of integration, then the integral on the left of \eqref{0eq: int of J = IK} turns into
\begin{align*}
4 \pi (-1)^k  \int_0^{2 \pi} \int_0^\infty  y^{4 \mu - 1} E \big(y e^{\frac 1 2 i \phi} \big)^{-2k}  \int_0^\infty  J_{2k} \big(  4 \pi \sqrt x Y \big(y e^{\frac 1 2 i \phi} \big) \big) \exp (- 2 \pi c x)  d x d y d\phi.
\end{align*}
Note that the triple integral is absolutely convergent as $|\Re \mu | < \frac 1 8$. We evaluate the inner integral using the formula \eqref{2eq: integral exp}, then the triple integral is equal to
\begin{equation*}
\begin{split}
\frac {  2 (- 2 \pi)^{k}  \Gamma \lp k + 1 \rp} {  \Gamma (2 k + 1) c^{ k + 1} }  \int_0^{2 \pi}  \int_0^\infty &  y^{4 \mu - 1}  \big(  y e^{- \frac 1 2 i \phi} + y\- e^{\frac 1 2 i \phi} \big)^{2k}  \\
& M  \lp k   + 1; 2k+1; -  \frac {2 \pi \lp y^2 + y^{-2} + 2 \cos \phi \rp} c \rp d y d \phi.
\end{split}
\end{equation*}

When $k = 0$,  Weber's confluent hypergeometric function reduces to the exponential function as in \eqref{2eq: M(a; a)}, so the double integral splits into a product of two integrals and they may be evaluated by \eqref{2eq: integral repn of K} and \eqref{2eq: integral repn of I} respectively. Consequently, we obtain
\begin{align}
\frac { 4 \pi } {  \hskip 1 pt  c \hskip 1 pt } \hskip 1 pt  K_{2 \mu} \lp \frac {4 \pi} {c} \rp I_0 \lp \frac {4 \pi  } {c} \rp.
\end{align}
then follows  the formula \eqref{0eq: int of J = IK} in Theorem \ref{thm: integral of J} in the case $k = 0$.

When $k \geqslant 1$, we apply the integral representation of Weber's  confluent hypergeometric function in \eqref{2eq: integral repn of M(a; b)},  expand $\big( y e^{- \frac 1 2 i \phi} + y\- e^{\frac 1 2 i \phi} \big)^{2k}$, change the order of integration, and again evaluate the integrals over $y$ and $\phi$ by \eqref{2eq: integral repn of K} and \eqref{2eq: integral repn of I} respectively. It follows that the integral above turns into
\begin{align*}
\frac { 4 \pi (- 2 \pi)^{k  }  } {    \lp k - 1 \rp ! c^{ k + 1} } \sum_{n = - k}^{k} (-)^n  C_{2k}^{k - n}  \int_0^1  K_{2 \mu +  n }    \lp \frac {4 \pi  \varv } c \rp I_{ n }   \lp \frac {4 \pi  \varv } c \rp    \varv^{k} (1-\varv)^{k - 1} d \varv.
\end{align*}
By Lemma \ref{lem: recurrence, case l = 0}, this is further equal to
\begin{align*}
\frac { 4 \pi (- 2 \pi)^{k  }  } {    \lp k - 1 \rp ! c^{ k + 1} } \sum_{n = 0}^{k}     C_{2k}^{k - n}   \sum_{r = 0}^{\left\lfloor   n/2 \right\rfloor}  2^{n-2r} D_n^r    \int_0^1 K_{2 \mu}^{(n-2r)}   \lp \frac {4 \pi  \varv } c \rp  I_{ n }   \lp \frac {4 \pi  \varv } c \rp  \varv^{k} (1-\varv)^{k - 1} d \varv.
\end{align*}
We now perform integration by parts. Since $|\Re \mu | < \frac 1 8$ and $\mu \neq 0$, in view of the series expansions of  $K_{2 \mu}$ and $I_n$  at zero (see (\ref{2def: series expansion of J}, \ref{2eq: I, K and J})), all the boundary terms at  $\varv = 0$ vanish. Moreover, we observe that $k - 1$ many differentiations are required to remove the zero of $(1 - \varv)^{k-1}$ at $\varv = 1$, so all the boundary terms  at $\varv = 1$ are zero except for one, which is
\begin{align}
 \frac { 4 \pi (-1)^k} {    c   }  K_{2 \mu} \lp \frac {4 \pi} {c} \rp   I_k \lp \frac {4 \pi  } {c} \rp.
\end{align}
On the other hand, the resulting integral after  integration by parts is
\begin{align}
\frac { 4 \pi (- 2 \pi)^{k  }  } {    \lp k - 1 \rp ! c^{ k + 1} } \int_0^1 K_{2 \mu} \lp \frac {4 \pi \varv} {c} \rp  S_k \lp \varv,  \frac {4 \pi} {c} \rp  d \varv,
\end{align}
with
\begin{align}\label{4eq: S k(v, a)}
S_k (\varv, a) = \sum_{n = 0}^{k}  (-)^n   C_{2k}^{k - n}   \sum_{r = 0}^{\left\lfloor   n/2 \right\rfloor}  (2/a)^{n-2r} D_n^r \shskip  \partial_{\varv}^{n-2r} \big( I_n (a \varv) P_k (\varv) \big) ,
\end{align}
and $P_k (\varv) = \varv^{k} (1-\varv)^{k - 1}$. Therefore, the formula \eqref{0eq: int of J = IK} in Theorem \ref{thm: integral of J} follows immediately from the vanishing of  $S_k (\varv, a)$. When $k = 1$, for instance, we see from \eqref{2eq: recurrence, I} that
\begin{align*}
S_1 (\varv, a) = 2 I_0(a \varv) \varv - (2/a) \lp a I_1' (a\varv) \varv + I_1 (a \varv) \rp = 0.
\end{align*}
In general, $S_k (\varv, a ) \equiv 0$ may be readily proven by the following recursive identity.

\begin{prop}\label{prop: recursive S(a, v)} Let $P_k (\varv) = \varv^k (1-\varv)^{k-1}$, $Q (\varv) = \varv  (1-\varv)$ and  $R (\varv) = 1 - 3\varv$. Let $S_k (\varv, a)$ be defined as in \eqref{4eq: S k(v, a)}. We have
	\begin{align*}
	a^2 S_{k+1} (\varv, a) = \ & -  4 \big(  Q (\varv) \lp   \partial_{\varv}^2 S_k (\varv, a) -  a^2  S_k (\varv, a) \rp  \\ & +  R (\varv)    \partial_{\varv} S_k (\varv, a)  +  (k^2-1)    S_k (\varv, a) -  (k-1) k \shskip    S_{k-1} (\varv, a) \big).
	\end{align*}
	
\end{prop}

Before   starting the proof of Proposition \ref{prop: recursive S(a, v)}, let us prove one more lemma on the derivatives of $P_k (\varv) = \varv^k (1-\varv)^{k-1}$ that will play a very crucial role afterwards. It shows how their derivatives for two consecutive $k$ are related simply by $Q (\varv) = \varv  (1-\varv)$ and  $R (\varv) =  1 - 3\varv$. This lemma combines perfectly with  Lemma \ref{lem: combinatorics 3} and they will work together  to finish the proof.

\begin{lem}\label{lem: PQR}
	Let $l  $ be a nonnegative integer.
	
	{\rm (1).} We have
	\begin{align}\label{4eq: PQR, 1}
		(k - l +2 ) & P_{k+1}^{(l)} = (k+2) P_k^{(l)} Q + l P_k^{(l-1)}   R + (k-1)  (l-1) l P_k^{(l-2)}.
	\end{align}
	It is understood that, when $l = 0, 1$, the terms containing $P_{k}^{(-1)}$, $P_{k}^{(-2)}$ 
	do not occur in the identities due to the appearance of the factors  $l$, $l-1$.
	
	{\rm (2).} We have 
	\begin{equation}\label{4eq: PQR, 2}
		\begin{split}
			(k-1) k (k-l) P_{k-1}^{(l)} = (k-2) P_k^{(l+2)} Q & + (2k-l-2) P_k^{(l+1)} R \\
			& + (k+1) (2k-l-2) (2k-l-1) P_k^{(l )} .
		\end{split}
	\end{equation}
\end{lem}
\begin{proof}
	First, we have
	\begin{align*}
		P_{k+1}' (\varv) = ((k+1) - (2k+1) \varv) P_{k} (\varv),
	\end{align*} 
	and therefore
	\begin{align}\label{4eq: derivative, 1}
		P_{k+1}^{(l+1)} (\varv) = ((k+1) - (2k+1) \varv) P_{k}^{(l )} (\varv) - (2k+1)l P_{k}^{(l-1)} (\varv).
	\end{align}
	We prove \eqref{4eq: PQR, 1} inductively. When $l = 0$, \eqref{4eq: PQR, 1} is simply $P_{k+1} (\varv) = P_k (\varv) Q (\varv)$. We now assume that  the identity \eqref{4eq: PQR, 1} is valid for $l$. By differentiating \eqref{4eq: PQR, 1} and then subtracting   \eqref{4eq: derivative, 1}, we obtain  the identity \eqref{4eq: PQR, 1} for $l+1$.
	
	Second, we have
	\begin{align*}
		(k-1) k P_{k-1} (\varv) = \lp (k-1) - (2k-1) \varv \rp P_k' (\varv)  + (2k-1)^2 P_k (\varv) ,
	\end{align*}
	and therefore
	\begin{align}\label{4eq: derivative, 2}
		(k-1) k P_{k-1}^{(l+1)} (\varv) = ((k-1) - (2k-1) \varv) P_{k}^{(l+2 )} (\varv) - (2k-1) (2k-l-2) P_{k}^{(l+1)} (\varv).
	\end{align}
	With \eqref{4eq: derivative, 2}, we may prove \eqref{4eq: PQR, 2} in the same fashion as \eqref{4eq: PQR, 1}.
\end{proof}

\subsection{Proof of Proposition \ref{prop: recursive S(a, v)}}
In the rest of this section, we shall exploit the combinatorial results established in \S \ref{sec: combinatorial}, along with Lemma \ref{lem: PQR}, to prove Proposition \ref{prop: recursive S(a, v)}. We would like to first get rid of the derivatives of $I_n$ in the expression of  $S_k (\varv, a)$ as in \eqref{4eq: S k(v, a)} by Lemma \ref{lem: recurrence, general case} and \ref{lem: combinatorics 2}. After this,  arise the combinatorial coefficients $A_{k,\shskip  l}^r$ defined in Lemma  \ref{lem: combinatorics 2}. Finally, the identities for  $A_{k,\shskip  l}^r$ and for the derivatives of $P_k$  in Lemma \ref{lem: combinatorics 3} and  \ref{lem: PQR} will be applied    to complete the proof.

First, applying Lemma \ref{lem: recurrence, general case} and \ref{lem: combinatorics 2},
we simplify $S_k (\varv, a)$ as follows,
\begin{align*}
S_k (\varv, a)  & =  \sum_{n = 0}^{k} (-)^n C_{2k}^{k - n} \sum_{r = 0}^{\left\lfloor   n/2 \right\rfloor}  2^{n-2r} D_{n }^r  \sum_{ l = 0}^{n-2r} C_{n-2r}^{l}  \shskip  (1/a)^{ l} I_{ n }^{(n-l-2r )} (a \varv) P_k^{(l)} (\varv) \\ 
& = \sum_{ l = 0}^{k } (1/a)^{ l} P_k^{(l)} (\varv) \sum_{n = l}^{k} (-)^n C_{2k}^{k - n}  \sum_{r = 0}^{\left\lfloor   (n-l)/2 \right\rfloor} 2^{n-2r} D_{n }^r   C_{n-2r}^{l} \shskip  I_{ n }^{(n-l-2r )} (a \varv) \\
& = \sum_{ l = 0}^{k } (2/a)^l P_k^{(l)} (\varv) \sum_{n = l}^{k} (-)^n C_{2k}^{k - n}  \sum_{r = 0}^{ n-l   } B_{l,\shskip  n }^r \shskip  I_{ l + 2 r } (a \varv) \\
& = \sum_{ l = 0}^{k } (2/a)^l P_k^{(l)} (\varv) \sum_{r =0}^{k-l} I_{l+2r} (a \varv) \sum_{n=l+r}^k (-)^n C_{2k}^{k - n}  B_{l,\shskip  n }^r   \\
& =   \sum_{ l = 0}^{k } (-2/a)^l P_k^{(l)} (\varv) \sum_{r =0}^{k-l} (-)^r A_{k,\shskip  l}^r I_{l+2r} (a \varv).
\end{align*}
Accordingly, we define  
\begin{align}\label{4eq: S kl}
S_{k,\shskip  l} (z) = \sum_{r =0}^{k-l} (-)^r A_{k,\shskip  l}^r I_{l+2r} (z),
\end{align}
so that
\begin{align}\label{4eq: S k}
S_{k} (\varv, a) = \sum_{ l = 0}^{k } (-2/a)^l P_k^{(l)} (\varv) S_{k,\shskip  l} (a \varv).
\end{align}
In view of our conventions on $A_{k,\shskip  l}^{r}$,  we shall put  $S_{k,\shskip  l} = 0$ when  either $k < l$ or $l < 0$. By $2 I'_{\nu} (z) = I_{\nu-1} (z) + I_{\nu+1} (z)$ and $4 I''_{\nu} (z) = I_{\nu-2} (z) + 2 I_{\nu} (z) + I_{\nu+2} (z)$ (see \eqref{2eq: recurrence, I and K}), we have
\begin{align}\label{4eq: S kl '}
2 S_{k,\shskip  l-1}' (z) = \sum_{r = -1}^{k - l + 1} (-)^r \big(  A_{k,\shskip  l-1}^r - A_{k,\shskip  l-1}^{r+1} \big) I_{l + 2 r } (z),
\end{align} 
\begin{align}\label{4eq: S kl ''}
4 S_{k,\shskip  l}'' (z) = - \sum_{r = -1}^{k - l + 1} (-)^r \big(  A_{k,\shskip  l}^{r-1} - 2 A_{k,\shskip  l}^{r } + A_{k,\shskip  l}^{r+1} \big) I_{l + 2 r } (z).
\end{align}

After these preparations, we are now ready to finish the proof. Consider 
\begin{align*}
S_{k+1}  +    Q \lp  (2/a)^2    \partial_{\varv}^2 S_k - 4   S_k \rp   +  R  (2/a)^2  \partial_{\varv} S_k  +  (k^2-1)  (2/a)^2 S_k ,
\end{align*}
and recall that our goal is to prove that it is equal to $ (k-1) k \shskip  (2/a)^2 S_{k-1} $. Using \eqref{4eq: S k}, straightforward calculations show that it may be partitioned into the sum of
\begin{equation}\label{4eq: partition, 4}
(-2 / a)^{k + 2 }   \lp  P_k^{(k+2)} Q +  P_k^{(k+1)} R +   (k^2-1)   P_k^{(k)}   \rp S_{k,\shskip  k},
\end{equation}
\begin{equation}\label{4eq: partition, 2}
P_{k+1} S_{k+1, 0} + 4 P_k Q \lp    S_{k, 0 }'' -  S_{k, 0 } \rp,
\end{equation}
\begin{equation}\label{4eq: partition, 3}
(- 2 / a ) \lp  P_{k+1}' S_{k+1, 1} - 4 P_k' Q \lp     S_{k,  0}' - S_{k, 1}''   + S_{k, 1}   \rp - 2 P_k R \shskip  S_{k,  0} ' \rp,
\end{equation}
\begin{equation}\label{4eq: partition, 1}
\begin{split}
\sum_{ l = 2}^{k + 1} (-2/a)^l \Big( & P_{k+1}^{(l)}  S_{k+1,\shskip  l}  
+  P_k^{(l)}  Q  \lp    S_{k,\shskip  l-2}  - 4   S_{k,\shskip  l-1}' + 4   S_{k,\shskip  l }''   - 4 S_{k,\shskip  l } \rp    \\
& +    P_k^{(l-1)} R    \lp   S_{k,\shskip  l-2} - 2 S_{k,\shskip  l-1} '   \rp 
+ (k^2-1)   P_k^{(l-2)} \shskip  S_{k,\shskip  l - 2}   \Big).
\end{split}
\end{equation}
In the above expressions, for succinctness, we have suppressed the arguments  $\varv$, $a$ and $a \varv$ from $S_k (\varv, a)$..., $P_k^{(l)} (\varv)$..., $Q (\varv)$, $R (\varv)$ and $S_{k,\shskip  l} (a \varv)$.... Next we apply Lemma  \ref{lem: combinatorics 3} and  \ref{lem: PQR}.

First, choosing $l = k+2$ in \eqref{4eq: PQR, 1} in Lemma \ref{lem: PQR}, it is clear that the sum in \eqref{4eq: partition, 4} vanishes. 
Second, computing with (\ref{4eq: S kl}, \ref{4eq: S kl '}, \ref{4eq: S kl ''}), Lemma \ref{lem: PQR} (1) in the cases $l = 0$ and $ 1$, Lemma \ref{lem: combinatorics 3} (1) and (2) imply that both  \eqref{4eq: partition, 2}	and \eqref{4eq: partition, 3} yield zero contribution. Note that, for $S_{k,\shskip 0}''$, $S_{k,\shskip 0}'$ and $S_{k,\shskip 1}''$  that occur in \eqref{4eq: partition, 2}	and \eqref{4eq: partition, 3}, one needs to combine the terms with $r = -1$ and $r = 1$ in \eqref{4eq: S kl '} or \eqref{4eq: S kl ''} by $I_{ - 1} (z) = I_{ 1} (z)$ and $I_{ - 2} (z) = I_{ 2} (z)$. Finally, from   Lemma \ref{lem: PQR} (1), Lemma \ref{lem: combinatorics 3} (3) and at last Lemma \ref{lem: PQR} (2), we infer that \eqref{4eq: partition, 1} is equal to
\begin{align*}
  -   (k-1) k \sum_{ l = 2}^{k + 1} (-2/a)^l P_{k-1}^{(l-2)} \sum_{r=-1}^{k-l +1} (-)^r    A^{r+1 }_{k-1 ,\shskip  l-2} I_{l+ 2 r} = (k-1) k \shskip  (2/a)^2 S_{k-1}. 
\end{align*}
	This completes the proof of Proposition \ref{prop: recursive S(a, v)} and hence Theorem \ref{thm: integral of J}.

\section{Proof of Theorem \ref{thm: main 1}}

Assume for simplicity that $\mu \neq 0$. Let $\rho = |\Re \mu | < \frac 1 2$.
We denote
\begin{align} 
\label{5eq: G (y)}
\bfG_{ \mu,\shskip  m } \lp y e^{i \theta} \rp & =  2 \int_{0}^{2 \pi} \int_0^\infty   \bfJ_{ \mu,\shskip  m }    \lp x e^{i\phi} \rp  e \lp - \frac { 2 x  \cos (\phi - \theta) } { y} \rp  d x  d \phi , \\
\label{5eq: F (y)}
& \bfF_{ \mu,\shskip  m } \lp y e^{i \theta} \rp  =  \frac { 2 } {\hskip 1 pt y \hskip 1 pt} e\lp -  y  \cos \theta   \rp \bfG_{ \mu,\shskip  m } \lp y e^{i \theta} \rp .
\end{align}

\subsection{Asymptotic of $\bfG_{\mu,\shskip  m} \lp y e^{i \theta} \rp$}\label{sec: asymptotic}

Our first task is  to prove the following asymptotic of $\bfG_{\mu,\shskip  m} \lp y e^{i \theta} \rp$,
\begin{equation}\label{3eq: asymptotic G (y)}
\bfG_{\mu,\shskip  m} \lp y e^{i \theta} \rp \sim      e \lp 2 y \cos \theta \rp + (-1)^{\frac 1 2 m} , \hskip 10 pt y \ra \infty,
\end{equation}
which yields
\begin{equation}\label{3eq: asymptotic F (y)}
\bfF_{\mu,\shskip  m} \lp y e^{i \theta} \rp \sim \frac { 2 } {\hskip 1 pt y \hskip 1 pt}     \lp e \lp   y \cos \theta \rp + (- )^{\frac 1 2 m} e\lp -  y  \cos \theta  \rp  \rp, \hskip 10 pt y \ra \infty.
\end{equation}

In the sequel, we shall fix a constant $ 0 < \delta < \frac 1 3$  and let $y  $ be sufficiently large. All the implied constants in our computations will only depend on $\delta$, $\rho$ and $ m $.

We split $\bfG_{\mu,\shskip  m} \lp y e^{i \theta} \rp$ as the sum
\begin{align*}
\bfG_{\mu,\shskip  m} \lp y e^{i \theta} \rp & =  \bfC_{\mu,\shskip  m} \lp y e^{i \theta} \rp + \bfD_{\mu,\shskip  m} \lp y e^{i \theta} \rp + \bfE_{\mu,\shskip  m} \lp y e^{i \theta} \rp \\
& = 2 \int_{0}^{2 \pi} \int_0^{4 y^{2 \delta}} u (x) \bfJ_{ \mu,\shskip  m }    \lp x e^{i\phi} \rp  e \lp - \frac { 2 x  \cos (\phi - \theta) } { y} \rp  d x d \phi  \\
& \hskip 13 pt + 2 \int_{0}^{2 \pi} \int_{ \frac 1 9 y^{2} }^{  9 y^2 }  w \big(    \sqrt {x } / y \big)   \bfJ_{ \mu,\shskip  m }    \lp x e^{i\phi} \rp  e \lp - \frac { 2 x  \cos (\phi - \theta) } { y} \rp d x d \phi \\
& \hskip 13 pt + 2 \int_{0}^{2 \pi} \int_{   y^{2 \delta} }^{\frac 1 4 y^2}  +   \int_{0}^{2 \pi} \int_{ 4 y^2 }^{\infty}   v \big(    \sqrt {x } / y \big)   \bfJ_{ \mu,\shskip  m }    \lp x e^{i\phi} \rp  e \lp - \frac { 2 x  \cos (\phi - \theta) } { y} \rp d x d \phi,
\end{align*}
where  $ u \big(y^2x^2 \big) + v (x) + w (x) \equiv 1 $
is a partition of unity on $ \lp 0 , \infty \rp$ such that $u (x) $, $v (x)  $ and $w (x) $ are smooth functions supported on $\lp 0, 4 y^{2 \delta  } \right]$,       $ \left[  y^{\delta - 1} , \tfrac 1 2 \right] \cup  \left[  2, \infty \rp$ and  $   \left[\frac 1 3, 3\right]   $ respectively, and that $x^r u^{(r)} (x), x^r v^{(r)}  (x), x^r w^{(r)}   (x) $ are bounded   for $r = 0, 1, 2$.

We shall first prove that $\bfC_{\mu,\shskip  m} \lp y e^{i \theta} \rp  \sim  (-1)^{\frac 1 2 m}$ as $y \ra \infty$ by the radial exponential integral formula \eqref{0eq: int of J = IK} in Theorem \ref{thm: integral of J}. A key observation is that $ e (- x \cos (\phi - \theta) /y) \sim \exp \lp - 2 \pi x /y  \rp \sim 1$ when $x \leqslant 4 y^{2 \delta}$. We shall then treat $\bfD_{\mu,\shskip  m} \lp y e^{i \theta} \rp$ and $ \bfE_{\mu,\shskip  m} \lp y e^{i \theta} \rp$ by the method of stationary phase. Since the stationary point is in the support of $w$ but outside of that of $v$, we may show that  the main term $\bfD_{\mu,\shskip  m} \lp y e^{i \theta} \rp \sim e (2 y \cos \theta)$ and that $ \bfE_{\mu,\shskip  m} \lp y e^{i \theta} \rp$ contributes as an error term. 
Note that the variable $x$ will be changed into $y^2 x^2$ in the defining integrals for $\bfD_{\mu,\shskip  m} \lp y e^{i \theta} \rp$ and $ \bfE_{\mu,\shskip  m} \lp y e^{i \theta} \rp$, so we have chosen to write a strange-looking $u \big(y^2 x^2 \big)$ in the partition of unity above. 

A fundamental problem of concern is the convergence of the integal in  \eqref{eq: main 1} or \eqref{5eq: G (y)}. 
The treatment of  $ \bfE_{\mu,\shskip  m} \lp y e^{i \theta} \rp$ in \S \ref{sec: Bound for E} will suggest that the integral should become absolutely convergent after one application of the elaborated partial integration of H\"ormander. 
We however feel that it is necessary to extract the main ideas and provide more details exclusively for the proof of the convergence. This will be done in Appendix \ref{appendix: convergence}.


\subsubsection{Asymptotic of $\bfC_{\mu,\shskip  m} \lp y e^{i \theta} \rp$}
When $x \leqslant 4 y^{2 \delta}$, we write $e (- x \cos (\phi - \theta) /y) = 1 + O (x/y)$. In view of the estimate in \eqref{2eq: bounds for J}, the contribution of the error term to $\bfC_{\mu,\shskip  m} \lp y e^{i \theta} \rp$ is bounded by
\begin{align*}
\frac 1 y \int_0^{2 \pi} \int_{0}^{1}   x^{- 2 \rho + 1 }  d x d \phi + \frac 1 y \int_0^{2 \pi} \int_{1}^{4 y^{2 \delta}}     \sqrt x   d x  d \phi = O \big( y^{  3   \delta - 1} \big).
\end{align*}
Therefore,
\begin{align*}
\bfC_{\mu,\shskip  m} \lp y e^{i \theta} \rp = 2 \int_0^{2 \pi} \int_{0}^{4 y^{2 \delta}} u (x) \bfJ_{ \mu,\shskip  m }  \lp x e^{i\phi} \rp d x d \phi +  O \big( y^{  3    \delta - 1} \big).
\end{align*}
On the other hand, choosing $c = 1/ y $ in Theorem \ref{thm: integral of J}, from the asymptotics of modified Bessel functions in \eqref{2eq: asymptotic I} and \eqref{2eq: asymptotic K} we infer that
\begin{align*}
2 \int_0^{2 \pi} \int_{0}^{\infty} \bfJ_{ \mu,\shskip  m }  \lp x e^{i\phi} \rp \exp \lp - \frac {2 \pi x}  y  \rp d x d \phi = (-1)^{\frac 1 2 m} + O \lp \frac 1 y \rp.
\end{align*}
With slight abuse of notation, we denote the double integral on the left by $\bfC_{\mu,\shskip  m} \lp y   \rp $, and split $\bfC_{\mu,\shskip  m} \lp y   \rp = \bfA_{\mu,\shskip  m} \lp y   \rp + \bfB_{\mu,\shskip  m} \lp y   \rp $ according to the partition of unity $u (x) + \lp 1 - u (x) \rp \equiv 1$. 
On writing $ \exp \lp - 2 \pi x /y  \rp = 1 + O \lp x /y  \rp$ for $x \leqslant 4 y^{2 \delta}$, similar as above, we find that
\begin{align*}
\bfA_{\mu,\shskip  m} \lp y   \rp = 2 \int_0^{2 \pi} \int_{0}^{4 y^{2 \delta}} u (x) \bfJ_{ \mu,\shskip  m }  \lp x e^{i\phi} \rp   d x d \phi + O \big( y^{  3   \delta - 1} \big).
\end{align*}
As for $ \bfB_{\mu,\shskip  m} \lp y   \rp$, we   insert the asymptotic of $\bfJ_{ \mu,\shskip  m }  \lp x e^{i\phi} \rp$ as in \eqref{2eq: asymptotic of J mu m}, with an error term contribution $ O \big( y^{-   \delta} \big) $. For the pair of leading terms, we need to consider
\begin{align*}
B  (y) =  \int_{y^{ 2 \delta}}^{\infty}     { \lp 1 - u (x) \rp \exp (- 2 \pi x / y)  }  \frac 1 {\sqrt x} \int_0^{4 \pi}   e \lp  4 \sqrt x \cos \lp \tfrac 1 2 \phi \rp \rp d \phi d x.
\end{align*} 
For the two pairs of  lower order terms, the treatments will be similar. In view of  \eqref{2eq: integral repn of J} and \eqref{2eq: asymptotic J}, we see that 
\begin{align*}
 B(y) & = 4 \pi \int_{y^{2  \delta}}^{\infty} \lp 1 - u (x) \rp   {\exp (- 2 \pi x  / y)} \frac 1 {\sqrt x}    J_0 (8 \pi \sqrt x) d x \\
& =   2 \int_{y^{2 \delta}}^{\infty} \lp 1 - u (x) \rp x^{- \frac 3 4}   {\exp (- 2 \pi x  / y)}      \cos \lp 8 \pi \sqrt x - \tfrac 1 4 \pi \rp d x + O \big( y^{- \frac 1 2 \delta} \big)  .
\end{align*}
Applying integration by parts to the oscillatory integral, we get
\begin{align*}
B (y) & =   \frac 1 {8 \pi y} \int_{y^{ 2 \delta}}^{\infty}   ( 1 - u (x) ) \big( {x^{- \frac 5 4}} y + 8 \pi x^{-\frac 1 4}   \big) \exp (- 2 \pi x  / y) \sin \lp 8 \pi \sqrt x - \tfrac 1 4 \pi \rp d x \\
& \hskip 13 pt + \frac 1 { 2 \pi} \int_{y^{ 2 \delta}}^{4 y^{ 2 \delta} } \hskip -2 pt u' (x)  x^{-\frac 1 4}   \exp (- 2 \pi x  / y) \sin \lp 8 \pi \sqrt x - \tfrac 1 4 \pi \rp d x + O \big( y^{- \frac 1 2 \delta} \big) \\
& =   O \big( y^{- \frac 1 2  \delta} + y^{\frac 3 2 \delta - 1} \big).
\end{align*}
Combining the foregoing results, we obtain
\begin{equation}\label{5eq: asymptotic of C}
\bfC_{\mu,\shskip  m} \lp y e^{i \theta} \rp  =  (-1)^{\frac 1 2 m} + O \big( y^{  3    \delta - 1} + y^{- \frac 1 2 \delta}\big).
\end{equation}

\subsubsection{Asymptotic of $\bfD_{\mu,\shskip  m} \lp y e^{i \theta} \rp$}\label{sec: asymptotic D}
Insert the asymptotic of    $\bfJ_{ \mu,\shskip  m }  \lp x e^{i\phi} \rp$ in \eqref{2eq: asymptotic of J mu m} into the integral that defines $\bfD_{\mu,\shskip  m} \lp y e^{i \theta} \rp$. The contribution from the error term is $O  \lp 1/y \rp$. For the three pairs of main terms, we change the variable of integration from $x e^{i \phi}$ to $y^2 x^2 e^{2 i \phi}$. Then we obtain an oscillatory integral from the pair of leading terms as below, along with two similar integrals from the two pairs of lower order terms,
\begin{align*}
D \lp y, \theta \rp   = 4 y \int_0^{2 \pi} \int_{\frac 1 3}^{3}   w (x)   e ( y f ( x, \phi; \theta ) )   d x d \phi, 
\end{align*}
with the phase function $f (x, \phi, \theta)$ defined by
\begin{align*}
f (x, \phi; \theta) = 4 x \cos \phi - 2 x^2 \cos (2 \phi - \theta) .
\end{align*}
We have
\begin{align*}
f' (x, \phi; \theta) = \lp 4 \lp \cos \phi -  x \cos (2 \phi - \theta) \rp,  - 4 x \lp \sin \phi  -  x  \sin (2 \phi - \theta) \rp \rp.
\end{align*}
Hence there is a unique stationary point $(x_0, \phi_0) = (1, \theta)$ on the domain of integration. Moreover, we have $f (x_0, \phi_0; \theta) = 2 \cos \theta$, $\det f'' (x_0, \phi_0; \theta) = - 16$ and $w (x_0) = 1$. Applying Lemma \ref{lem: stationary phase}, we obtain
\begin{align*}
D (y, \theta) = e (2 y \cos \theta)+ O \lp 1/ y  \rp. 
\end{align*}
Similarly, we find that two other integrals of lower order are both $O (1/y)$.
Therefore,
\begin{align}\label{5eq: asymptotic of D}
\bfD_{\mu,\shskip  m} \lp y e^{i \theta} \rp = e (2 y \cos \theta)+ O \lp 1/ y  \rp.
\end{align}

 \subsubsection{Bound for $\bfE_{\mu,\shskip  m} \lp y e^{i \theta} \rp$}\label{sec: Bound for E}
Similar as in \S \ref{sec: asymptotic D}, from the asymptotic of $\bfJ_{ \mu,\shskip  m }  \lp x e^{i\phi} \rp$, the error term contributes $ O \big( y^{-  \delta} \big)$ and we need to consider the following oscillatory integral,
\begin{align*}
E \lp y, \theta \rp  & = 4 y \int_0^{2 \pi} \int_{y^{\delta - 1}}^{\infty}   v (x)   e ( y f (  x, \phi; \theta ) )   d x d \phi,
\end{align*}
and two other similar integrals. As the stationary point $(x_0, \phi_0) = (1, \theta)$ does not lie in the support of $v (x)$, according to the method of stationary phase, roughly speaking, one should get a saving of $y$ for each partial integration. Our goal is to show that $E \lp y, \theta \rp = o (1)$ after twice of partial integrations. 
For this, we define
\begin{equation*}
g (x, \phi; \theta) = x^{-1} \lp \partial_{x} f(x, \phi; \theta) \rp^2 + x^{-3} \lp \partial_\phi f(x, \phi; \theta) \rp^2 = 16 \lp x  + x^{-1} - 2 \cos (\phi - \theta) \rp.
\end{equation*}
An important observation is that, on  the support of $v (x)$, we have 
\begin{equation*}
	g (x, \phi; \theta) \geqslant 4 \max \left\{ 1/x , x  \right\} \text{(}\geqslant 8\text{)}, \hskip 10 pt  x \in \left[  y^{\delta - 1}, \tfrac 1 2 \right] \cup \left[ 2, \infty \rp.
\end{equation*}
Besides, the following simple upper bounds  will also be useful
\begin{align*}
& (1/x)   \partial_\phi^r f  (x, \phi; \theta),  \partial_\phi^r \partial_x f (x, \phi; \theta) \lll \max \big\{ 1, x \big\}, \hskip 10 pt \partial_\phi^r \partial_x^2  f (x, \phi; \theta) \lll 1, \hskip 10 pt \partial_x^3 f (x, \phi; \theta) \equiv 0, \\
& \partial_\phi, \partial_\phi^2 g  (x, \phi; \theta)  \lll 1,    \partial_x g (x, \phi; \theta) \lll \max \left\{ 1/x^2 , 1  \right\}, \partial_x^2 g  (x, \phi; \theta) \lll 1/x^3,   \partial_x \partial_\phi g (x, \phi; \theta) \equiv 0.
\end{align*} 
We now apply the elaborated partial integration of H\"ormander (see the proof of \cite[Theorem 7.7.1]{Hormander}\footnote{Our ideas in spirit are contained in the proof of H\"ormander's Theorem 7.7.1. Indeed,  $g (x, \phi; \theta) $ plays the same role as $|f'|^2 + \Im f$ in his proof, where $f$ is the phase function therein. His  Theorem 7.7.1 however does not apply here as our domain of integration is {\it non-compact}. Also, strictly speaking, our $g (x, \phi; \theta) $ is constructed differently (note the factors $x^{-1}$ and $x^{-3}$ in its definition) in an {\it ad hoc} manner.}). Writing
\begin{align*}
E \lp y, \theta \rp & =   \frac {2}{\pi i}  \int_0^{2 \pi} \int_{y^{\delta - 1}}^{\infty} \frac {v (x)} {g (x, \phi; \theta)} \bigg( \frac {  \partial_{x} f(x, \phi; \theta)} {x } \partial_x (   e (  y f (  x, \phi; \theta ) ) )   \\
& \hskip 97 pt  +   \frac {   \partial_{\phi} f(x, \phi; \theta)} {x^3  } \partial_\phi (   e (  y f (  x, \phi; \theta ) ) ) \bigg)  d x d \phi,
\end{align*}
and integrating by parts, then $E \lp y, \theta \rp  $ turns into
\begin{align*} 
&   -  \frac {2}{\pi i} \int_0^{2 \pi} \int_{y^{\delta - 1}}^{\infty}   \lp \hskip -1 pt \frac {\partial} {\partial x} \lp \frac {v (x) \partial_{x} f(x, \phi; \theta)} {x g (x, \phi; \theta)}  \rp \hskip -2 pt +  \frac {v (x)}{x^3} \frac {\partial} {\partial \phi} \lp \frac { \partial_{\phi} f(x, \phi; \theta)} { g (x, \phi; \theta)}    \rp  \hskip -1 pt \rp e ( y f ( x, \phi; \theta ) ) d x d \phi.
\end{align*}
We then calculate the derivatives in the integrand by  the product rule for differentiations and  estimate each resulting integral by  the  bounds for $g (x, \phi; \theta) $ and $f (x, \phi; \theta)$ as above.  
It is trivially bounded by $\log y$. Indeed, the absolute value of each integrand is bounded by either $1/ x$ or $1$ for  $x \in \left[ y^{ \delta - 1}, \frac 1 2 \right]$ and by either $1/x^2$ or $1/x^3$ for $ x \in \left[2, \infty\right)$. In particular, all the integrals are absolutely convergent. Some details may be found in Appendix \ref{appendix: convergence}.

Applying  partial integration once more yields an additional saving of $y$. 
Therefore
\begin{align*}
E \lp y, \theta \rp = O (1/y).
\end{align*}
Moreover, one application of H\"ormander's elaborated partial integration on the two integrals of lower order is sufficient to yield the bound $1/y$. Hence
\begin{align}\label{5eq: bound for E}
\bfE_{\mu,\shskip  m} \lp y e^{i \theta} \rp = O \lp y^{- \delta}   \rp.
\end{align}


\subsubsection{Conclusion} Combining   (\ref{5eq: asymptotic of C}, \ref{5eq: asymptotic of D}, \ref{5eq: bound for E}), the proof of \eqref{3eq: asymptotic G (y)} is now complete.

\subsection{Differential Equations for $\bfF_{ \mu,\shskip  m }  $ }\label{sec: differential equations}
	We are now going to verify
	\begin{align}\label{5eq: differential equations}
	\nabla_{ \mu + \frac 1 4 m} \lp \bfF_{ \mu,\shskip  m } (u / \pi) \rp = 0, \hskip 10 pt \overline \nabla_{ \mu - \frac 1 4 m} \lp \bfF_{ \mu,\shskip  m } (u / \pi) \rp = 0.
	\end{align}
	By symmetry, we only need to verify the former, which may be explicitly written as
	\begin{align}\label{5eq: differential equations, explicit}
	u^2 \frac {\partial^2 \bfF_{ \mu,\shskip  m }} {\partial u^2} (u) + u \frac {\partial  \bfF_{ \mu,\shskip  m }} {\partial u} (u) + \lp \pi^2 u^2 -   \lp \mu + \frac 1 4 m \rp^2 \rp \bfF_{ \mu,\shskip  m } (u) = 0.
	\end{align}
	
	Before proceeding to the calculations, we remark that the differentiations of  $\bfF_{ \mu,\shskip  m }$ (under the integral) should be interpreted   in the theory of distributions. This is explained as follows. For any test function $ f (u) \in  \SS (\BC \smallsetminus \{0\} )$, let
	\begin{align*}
	\left\langle \bfF_{ \mu,\shskip  m }  , f   \right\rangle = 2 \sideset{}{_{\BC \smallsetminus \{0 \}}}{\iint}  \bfJ_{ \mu, \shskip  m} (z) f^{\sharp} (z) \frac { i d z \nwedge d \overline z} {|z| },
	\end{align*}
	with $ f^{\sharp} (z) \in  \SS (\BC  )$ given by
	\begin{align*}
	f^{\sharp} (z) = \sideset{}{_{\BC \smallsetminus \{0 \}}}{\iint}  \frac 1 {\shskip|u |}  e \hskip -1 pt \lp -  \Tr \lp \frac {u  } {2  }    +  \frac z {u} \rp \rp  f (u)   \shskip {i d u \nwedge d \overline u}  . 
	\end{align*}
	Note here that $ f^{\sharp} (z)$ is the Fourier transform of $ e (- \Tr (1/2 u)) f (1/u) / |u|^3 $. According to the theory of distributions, 
	\begin{align*}
	 \left\langle  \partial \bfF_{ \mu,\shskip  m }  /\partial u , f \right\rangle = - \left\langle  \bfF_{ \mu,\shskip  m }, \partial f  /\partial u  \right\rangle = 2 \sideset{}{_{\BC \smallsetminus \{0 \}}}{\iint}  \bfJ_{ \mu, \shskip  m} (z) f^{\natural} (z) \frac { i d z \nwedge d \overline z} {|z| },
	\end{align*}
	with  $ f^{\natural} (z) \in  \SS (\BC  )$ given by 
	\begin{align*}
	f^{\natural} (z) & = - \sideset{}{_{\BC \smallsetminus \{0 \}}}{\iint}   \frac 1 {\shskip|u|} e \hskip -1 pt \lp -  \Tr \lp \frac {u  } {2  }    +  \frac z {u} \rp \rp \frac  {\partial f (u)} {\partial u}    {i d u \nwedge d \overline u}   \\
	& = \sideset{}{_{\BC \smallsetminus \{0 \}}}{\iint}  \frac  {\partial } {\partial u} \lp \frac 1 {\shskip|u|} e \hskip -1 pt \lp -  \Tr \lp \frac {u  } {2  }    +  \frac z {u} \rp \rp \rp f (u)  \shskip {i d u \nwedge d \overline u} .
	\end{align*} 
	In view of these,  differentiating  $\bfF_{ \mu,\shskip  m }$ under the integral is well justified in the theory of distributions. So are the formal calculations in what follows. 
	
	For $s  , r = 0, 1$, $2$, with $s + r = 0, 1, 2$, we introduce
	\begin{align*}
	\bfF_{s ,\shskip  r,\shskip  \mu, \shskip  m} (u) =   \frac 2 {\sqrt { u \overline u}} e \bigg( \hskip -2 pt -     \frac {u + \overline u} {2  }   \bigg) \sideset{}{_{\BC \smallsetminus \{0 \}}}{\iint} \hskip - 2 pt z^{s  +  r - \frac 1 2} \overline {z}^{- \frac 1 2} (\partial/\partial z)^{r} \bfJ_{ \mu, \shskip  m} (z) e \bigg( \hskip - 2 pt -     \frac z {u} - \frac {\overline z} {\overline u}     \bigg) i d z \nwedge d \overline z.
	\end{align*}
	Here $\bfF_{s ,\shskip  r,\shskip  \mu, \shskip  m} $ are regarded as distributions on $\BC \smallsetminus \{0\}$ as explained as above. 
	Note that $ \bfF_{  \mu, \shskip  m} = \bfF_{0 ,\shskip  0,\shskip  \mu, \shskip  m}  $.
	For brevity, we put $\bfF_{s ,\shskip  r}  = \bfF_{s ,\shskip  r,\shskip  \mu, \shskip  m}$ and $\bfF_{s } = \bfF_{s , \shskip 0}$.

	First, for $s  = 0, 1$, we have
	\begin{align*}
	\frac {\partial \bfF_{s  }} {\partial u} = - \lp   \pi i   + \frac 1 {2 u} \rp \bfF_{s  } + \frac {2 \pi i} {u^2} \bfF_{s +1 }
	\end{align*}
	and
	\begin{align*}
	\frac {\partial^2 \bfF_{0 }} {\partial u^2} & = \frac 1 {2 u^2} \bfF_{0  } -  \lp   \pi i + \frac 1 {2 u} \rp \frac {\partial \bfF_{0  }} {\partial u} - \frac {4 \pi i} {u^3} \bfF_{ 1   } + \frac {2 \pi i} {u^2} \frac {\partial \bfF_{1  }} {\partial u} \\
	& = \frac 1 {2 u^2} \bfF_{0  } - \lp   \pi i + \frac 1 {2 u} \rp \lp - \lp   \pi i + \frac 1 {2 u} \rp \bfF_{0  } + \frac {2 \pi i} {u^2} \bfF_{1  } \rp \\
	&  \hskip 11 pt  - \frac {4 \pi i} {u^3} \bfF_{1  } +  \frac {2 \pi i} {u^2} 
	\lp - \lp   \pi i + \frac 1 {2 u} \rp \bfF_{1  } + \frac {2 \pi i} {u^2} \bfF_{2 } \rp \\
	& = \lp - \pi^2 + \frac {\pi i} u + \frac {3} {4 u^2} \rp \bfF_{0  } + \lp  \frac {4 \pi^2} {u^2} - \frac {6 \pi i} {u^3} \rp \bfF_{1  } - \frac {4 \pi^2} {u^4} \bfF_{2 }.
	\end{align*}
	Second, for $s, r  = 0 , 1 $, with $s + r = 0, 1$,
	by partial integration, 
	\begin{align*}
	\frac {4 \pi i} u    \sideset{}{_{\BC \smallsetminus \{0 \}}}{\iint} & \hskip - 2 pt z^{s  +  r + \frac 1 2} \overline {z}^{- \frac 1 2} (\partial/\partial z)^{r} \bfJ_{  \mu, \shskip  m} (z) e \bigg( \hskip -2 pt - \hskip -1 pt \bigg( \frac z {u} + \frac {\overline z} {\overline u} \bigg) \bigg) i d z \nwedge d \overline z \\ 
	& =     \lp  2 s  + 2  r + 1 \rp   \sideset{}{_{\BC \smallsetminus \{0 \}}}{\iint} \hskip - 2 pt z^{s  +  r - \frac 1 2} \overline {z}^{- \frac 1 2} (\partial/\partial z)^{r} \bfJ_{  \mu, \shskip  m} (z) e \bigg( \hskip -2 pt - \hskip -1 pt \bigg( \frac z {u} + \frac {\overline z} {\overline u} \bigg) \bigg) i d z \nwedge d \overline z \\
	& \hskip 45 pt + 2 \sideset{}{_{\BC \smallsetminus \{0 \}}}{\iint} \hskip - 2 pt z^{s  +  r + \frac 1 2} \overline {z}^{- \frac 1 2} (\partial/\partial z)^{r+1} \bfJ_{ \mu, \shskip  m} (z) e \bigg( \hskip -2 pt - \hskip -1 pt \bigg( \frac z {u} + \frac {\overline z} {\overline u} \bigg) \bigg) i d z \nwedge d \overline z.
	\end{align*}
	It follows that 
	\begin{align*}
	  \frac {4 \pi i} u \bfF_{s +1,\shskip  r } =  \lp 2  s  +  2 r + 1 \rp \bfF_{s ,\shskip  r } + 2 \bfF_{s ,\shskip  r+1 }, 
	\end{align*}
	and hence
	\begin{align*}
	- \frac {16 \pi^2} {u^2} \bfF_{2 }   =   \frac {12 \pi i  } u     \bfF_{1  } + \frac {8 \pi i} u \bfF_{1 ,\shskip   1 }   =  3   \bfF_{0  }  +  12     \bfF_{0 ,\shskip   1 } + 4 \bfF_{0 ,\shskip   2 }.
	\end{align*}
	Third, since $\nabla_{2 \mu + \frac 1 2 m} \lp \bfJ_{\mu,\shskip  m} \big(   z^2 /16 \pi^2 \big) \rp = 0$ (see \eqref{2eq: nabla J = 0}), we have
	\begin{align*}
	4 \bfF_{0 ,\shskip  2 } + 4 \bfF_{0 ,\shskip  1 } + 16 \pi^2 \bfF_{ 1 ,\shskip  0 } - \lp 2 \mu + \frac 1 2 m \rp^2 \bfF_{0 ,\shskip  0 } = 0.
	\end{align*}
	Finally, combining these, we have
	\begin{align*}
	& \hskip 13 pt u^2 \frac {\partial^2 \bfF_{0  }} {\partial u^2} + u \frac {\partial  \bfF_{0  }} {\partial u } + \lp \pi^2 u^2 - \lp \mu + \frac 1 4 m \rp^2 \rp \bfF_{0  } \\
	& = \lp - \pi^2 u^2 +  \pi i  u + \frac {3} {4  } \rp \bfF_{0  } + \lp  {4 \pi^2} - \frac {6 \pi i} {u } \rp \bfF_{1  } - \frac {4 \pi^2} {u^2} \bfF_{2 }\\
	& \hskip 10 pt - \lp   \pi i u + \frac 1 {2  } \rp \bfF_{0  } +  \frac {2 \pi i} {u } \bfF_{1  } + \lp \pi^2 u^2 -   \lp \mu + \frac 1 4 m \rp^2 \rp \bfF_{0  } \\
	& = \lp \frac 1 4 - \lp \mu + \frac 1 4 m \rp^2 \rp \bfF_{0  } + \lp   {4 \pi^2}  - \frac {4 \pi i} {u } \rp \bfF_{1  } - \frac {4 \pi^2} {u^2} \bfF_{2 }\\
	& = \lp \frac 1 4 - \lp \mu + \frac 1 4 m \rp^2 \rp \bfF_{0  }  + {4 \pi^2}   \bfF_{1  } -  \bfF_{0  } - 2 \bfF_{0 ,\shskip   1 }  + \frac 3 4 \bfF_{0  } +  3  \bfF_{0 ,\shskip  1 } +  \bfF_{0 ,\shskip  2 } \\
	& = - \lp \mu + \frac 1 4 m \rp^2   \bfF_{0  } + {4 \pi^2}   \bfF_{1  } +   \bfF_{0 ,\shskip  1 } +   \bfF_{0 ,\shskip  2 }\\
	& = 0,
	\end{align*} 
	which proves \eqref{5eq: differential equations, explicit}.
	
	\subsection{Conclusion}
	Combining \eqref{3eq: asymptotic F (y)} and \eqref{5eq: differential equations},  Lemma \ref{lem: Bessel equation} implies that 
	\begin{align}\label{5eq: final identity}
	\bfF_{ \mu,\shskip  m } \lp 4 u \rp = \bfJ_{\frac 1 2 \mu,\shskip  \frac 1 2 m } \lp u^2 \rp.
	\end{align}
	In view of the definition of $\bfF_{ \mu,\shskip  m } $  given by (\ref{5eq: G (y)}, \ref{5eq: F (y)}), this is equivalent to the identity  \eqref{eq: main 1} in  Theorem \ref{thm: main 1}. 

	\section{Proof of Corollary \ref{cor: main}}\label{sec: proof of Cor}
	
	In this last section, we outline a proof of Corollary \ref{cor: main}. 

	Some remarks on the   convergence of \eqref{eq: main 2} for $f \in \SS (\BC)$ are in order. In view of the asymptotics of $\bfJ_{\mu,\shskip  m} \lp z \rp$ at zero and infinity, the right hand side of \eqref{eq: main 2} absolutely converges for all $\mu$ and $m$, but the absolute convergence of left hand side only holds for   $|\Re \mu | < \frac 1 2$. 
	
	To prove \eqref{eq: main 2} for $f \in \SS (\BC)$ in Corollary \ref{cor: main}, it suffices to verify the identity,
	\begin{equation}\label{3eq: identity}
	\begin{split}
	\int_0^{2 \pi} \int_0^\infty \int_0^{2 \pi}  \int_0^\infty \bfJ_{\mu,\shskip  m} & \lp x e^{i\phi} \rp     e (- 2 x y \cos (\phi + \theta) ) f \lp y e^{i\theta} \rp     y  d y d \theta d x d \phi  \\
	& = \frac 1 4  \int_0^{2 \pi} \int_0^\infty
	e \lp \frac {\cos \theta } y \rp \bfJ_{\frac 1 2 \mu,\shskip  \frac 1 2 m} \lp  \frac 1 { 16 y^2 e^{2 i \theta}  } \rp f \lp y e^{i\theta} \rp  d y d \theta.
	\end{split}
	\end{equation}
	Note that \eqref{3eq: identity} is a direct consequence of the identity \eqref{eq: main 1} in Theorem \ref{thm: main 1} if one were able to change the order of integrations. However, the decay of $\bfJ_{\mu,\shskip  m} (z)$ at infinity is too slow to guarantee absolute convergence and the change of integration order. In order to get absolute convergence, our idea is to produce some decaying factor by partial integrations. 
	
	Since the idea  is straightforward, we shall only give a sketch of the proof as below and leave the details to the readers. Starting from the integral on the left hand side of \eqref{3eq: identity}, we write the inner integral in the Cartesian coordinates 
	and apply the combination of partial integrations that have the effect of dividing $ 4 \pi^2 x^2 + 1$. To be precise, letting $z = x e^{i\phi} $ and $u = y e^{i \theta}$, define the differential operator $\mathrm{D} = - (\partial / \partial u) (\partial / \partial \overline u) + 1$ so that $\mathrm{D} \lp e (- \Tr (z u)) \rp = (   4 \pi^2 z \overline z + 1 ) e (- \Tr (z u))$, then, by partial integrations,   the left side of \eqref{3eq: identity} is equal to
	\begin{align*}
	\int_0^{2 \pi} \int_0^\infty  \frac { \bfJ_{\mu,\shskip  m} \lp x e^{i\phi} \rp }    { 4 \pi^2 x^2 + 1} \lp \frac 1 2  \sideset{ }{_{\BC \smallsetminus \{0\}} }{\iint}      e \lp - \Tr \lp x e^{i \phi} u \rp \rp    \mathrm{D}     {f (u) }     {i d u  \nwedge  d \overline u}  \rp d x d \phi .
	\end{align*}
	After this, we change the order of integrations, which  is legitimate as $  \bfJ_{ \mu,\shskip  m } \lp x e^{i\phi} \rp / ( 4 \pi^2 x^2 + 1 ) $ is absolutely integrable. Then, we apply the partial integrations reverse to those that we performed at the beginning and rewrite the outer integral in the polar coordinates. In this way, we retrieve the expression on the  left hand side of \eqref{3eq: identity} but with changed integration order. Here, to prove that   differentiations under the integral sign are permissible, we need   the compact convergence of the double integral in \eqref{eq: main 1}  with respect to $u = y e^{i \theta}$, but this is verified in Proposition \ref{lem: compact convergence}.  
	Finally,  integrating out the resulting inner integral using \eqref{eq: main 1} in Theorem \ref{thm: main 1}, we arrive at the integral on the right hand side of  \eqref{3eq: identity}.

	When $\widehat{f} \in \SS (\BC \smallsetminus \{0\})$, the integral on the left hand side of \eqref{eq: main 2} becomes absolutely convergent for all $\mu$, then follows the second assertion in Corollary \ref{cor: main}.

	\appendix
	
	\section{Convergence of the Integral}\label{appendix: convergence}
	
	As alluded to in \S \ref{sec: asymptotic}, the integral in  \eqref{eq: main 1} would be absolutely convergent after integration by parts. In this appendix, we shall make it more precise and in the meanwhile prove that the convergence is compact  (uniform convergence in compact subset). This was used in the proof of Corollary \ref{cor: main} in \S \ref{sec: proof of Cor}.

	\begin{prop}\label{lem: compact convergence}
		Suppose that $|\Re \mu| < \frac 1 2$ and $m$ is even. 
		The integral
		\begin{align}\label{6eq: integral}
		\int_{0}^{2 \pi} \int_0^\infty \bfJ_{ \mu,\shskip  m } \hskip - 2 pt \lp x e^{i\phi} \rp  e (- 2 x y \cos (\phi + \theta) )    d x  d \phi  
		\end{align}
		is compactly convergent with respect to $y e^{i \theta}$.
	\end{prop}
	
	\begin{proof}
		Fix $Y > 1$. We need to verify that the integral \eqref{6eq: integral} converges uniformly on the annulus  $\left\{ y e^{i \theta} : y \in \left[1/ Y, Y\right] \right\}$. In order to use the arguments in \S \ref{sec: asymptotic}, let us change $y$ to $1/y$ and $\theta$ to $- \hskip 1pt \theta$. 
		
		First, we make the change of variables from $x e^{i \phi}$ to $y^2 x^2 e^{2 i \phi}$. As  $ y \in \left[1/ Y, Y\right]$, this does not affect the uniformity of convergence. Now we need to consider 
		\begin{align*}
	  4 y^2 \int_{0}^{\pi} \int_0^\infty   \bfJ_{ \mu,\shskip  m }    \big( y^2 x^2 e^{2i\phi} \big)  e \big( -   { 2 x^2 y  \cos (2\phi - \theta) }  \big)  x \shskip d x  d \phi
		\end{align*}
		 Second, we introduce a smooth partition of unity $ \lp 1 - v (x) \rp + v (x) \equiv 1$ on $(0, \infty) =  (0, 3   ] \cup  [ 2  , \infty )$ and split the integral accordingly into two parts.   In view of the estimate for $ \bfJ_{ \mu,\shskip  m } (z)$ in   \eqref{2eq: bounds for J} and the condition $|\Re \mu| < \frac 1 2$,   the first integral is obviously uniformly convergent for $y \in [1/Y, Y]$. As for the second integral, 
		 we insert the asymptotic formula of    $\bfJ_{ \mu,\shskip  m }  \lp y^2 x^2 e^{2 i\phi} \rp$ as in \eqref{2eq: asymptotic of J mu m}. Note here that $y \geqslant 1/ Y$ is bounded from below. The integral containing the error term is absolutely and uniformly convergent. We then need to consider the following integral obtained from the pair of leading terms,
		 \begin{align*}
		  2 y \int_0^{2 \pi} \int_{2}^{\infty}   v (x)   e ( y f (  x, \phi; \theta ) )   d x d \phi,
		 \end{align*}
		 and two other similar integrals of lower order; the phase function $ f \lp x, \phi; \theta \rp $ is defined in \S \ref{sec: asymptotic D}. Applying once the partial integration of H\"ormander  as in \S \ref{sec: Bound for E}, we arrive at 
		 \begin{align*}    
		 -  \frac {1}{\pi i} \int_0^{2 \pi} \int_{2}^{\infty} \hskip -2 pt \lp   \frac {\partial} {\partial x} \lp \frac {v (x) \partial_{x} f(x, \phi; \theta)} {x g (x, \phi; \theta)}  \rp \hskip -2 pt +  \frac {v (x)}{x^3} \frac {\partial} {\partial \phi} \lp \frac { \partial_{\phi} f(x, \phi; \theta)} { g (x, \phi; \theta)}  \hskip - 1 pt  \rp \hskip -1 pt \rp e ( y f ( x, \phi; \theta ) ) d x d \phi.
		 \end{align*}
		 By  the product rule for differentiations, the function in the large parenthesis is equal to
		 \begin{align*}
		 {v ' \partial_{x} f } / {x g } +
		 {v   \partial_{x}^2 f } / {x g } -
		 {v  \partial_{x} f } / {x^2 g } -
		 {v  \partial_{x} f \partial_{x} g} / {x g^2 } +
		 {v  \partial_{\phi}^2 f } / {x^3 g } -
		 {v  \partial_{\phi} f \partial_{\phi} g} / {x^3 g^2 }.
		 \end{align*}
		 Recall from \S \ref{sec: Bound for E} that for $ x \geqslant 2$ 
		 \begin{equation*}
		 g (x, \phi; \theta) \ggg x, 
		 \end{equation*}
		 \begin{align*}
		 &   \partial_x f (x, \phi; \theta) \lll x, \hskip 8 pt 
		 \partial_x^2  f (x, \phi; \theta) \lll 1, \hskip 8 pt 
		 \partial_\phi,  \partial_\phi^2 f  (x, \phi; \theta) \lll x^2, \hskip 8 pt
		 \partial_x ,  \partial_\phi g (x, \phi; \theta) \lll 1 .
		 \end{align*} 
		 It is easy to prove by these bounds that all the terms in the sum above are $O \big(1/x^2 \big)$ (the first term is actually compactly supported while the last term is indeed $O \big(1/x^3 \big) $). Then follows immediately the absolute and uniform convergence of the integral.
	\end{proof}

\begin{acknowledgement}
	This work was done during my stay at Rutgers University. I would like to acknowledge the Department of Mathematics for the hospitality and thank Stephen D. Miller and Henryk Iwaniec for their help.  I am especially grateful to Roman Holowinsky and Jim Cogdell for their comments  on this work and, more importantly, constant
	encouragements. I also thank the referee for constructive remarks and suggestions.
\end{acknowledgement}

	%
	
	\bibliographystyle{alphanum}
	\bibliography{references}
	

\end{document}